\documentclass[12pt,french,american]{article}
\usepackage[total={7in,8.75in}]{geometry}
\usepackage{mathrsfs, bbm, dsfont}
\usepackage{
amsmath,
amsfonts,
latexsym, 
amsrefs,
amssymb,
mathtools,
tikz,
tabularx,
minibox,
amsthm,
}
\usepackage{bm}
\usepackage{babel}
\usepackage{floatrow}
\usepackage{comment}
\usepackage[T1]{fontenc}
\usepackage{graphicx}
\DeclareGraphicsExtensions{.eps}
\usepackage{url}
\usepackage[overload]{empheq} 
\usepackage{wrapfig}
\usepackage{titlesec}
\usepackage{titletoc}
\usepackage[title,titletoc]{appendix}
\usepackage{pdflscape}
\usepackage{afterpage}
\usepackage{tocloft}

\usepackage{caption}
\usepackage{subcaption}
\usepackage{enumerate}

\usepackage{todonotes}

\definecolor{mygray}{gray}{0.3} 
\usepackage[colorlinks=true,urlcolor=gray,linkcolor=mygray,citecolor=mygray]{hyperref}

\usetikzlibrary{shapes}
\usetikzlibrary{intersections}
\usetikzlibrary{svg.path}
\usetikzlibrary{decorations.markings}
\usetikzlibrary{hobby}

\numberwithin{equation}{section}
\renewcommand{\rm}{\mathrm}

\newcommand{\1}{\mathds{1}}

\newcommand{\eps}{\varepsilon}

\newcommand{\rd}{{\rm d}}

\newcommand{\al}{\alpha}

\newcommand{\e}{{\varepsilon}}

\newcommand{\la}{\lambda}

\newcommand{\OO}{\mathrm{O}}
\newcommand{\oo}{\mathrm{o}}
\newcommand{\ovS}{\overline{S}}
\newcommand{\ovP}{\overline{\PP}}
\newcommand{\ovE}{\overline{\E}}

\newcommand{\re}{\text{Re } }
\newcommand{\ii}{{\rm i}}

\DeclareMathOperator{\PP}{\mathbb{P}}
\DeclareMathOperator{\wPP}{\overline{\mathbb{P}}}
\renewcommand{\P}{\mathbb{P}}
\newcommand{\E}{\mathbb{E}}
\newcommand{\wE}{\overline{\mathbb{E}}}
\newcommand{\R}{\mathbb{R}}
\newcommand{\C}{\mathbb{C}}

\newcommand{\Z}{\mathbb{Z}}

\newcommand{\dd}{\mathrm{d}}

\theoremstyle{plain} 
\newtheorem{theorem}{Theorem}[section]
\newtheorem*{theorem*}{Theorem}
\newtheorem{lemma}[theorem]{Lemma}

\newtheorem*{lemma*}{Lemma}
\newtheorem{corollary}[theorem]{Corollary}
\newtheorem*{corollary*}{Corollary}
\newtheorem{proposition}[theorem]{Proposition}
\newtheorem*{proposition*}{Proposition}

\newtheorem*{fact*}{Fact}
\newtheorem*{conj}{Conjecture}

\newtheorem*{definition*}{Definition}

\newtheorem*{example*}{Example}
\newtheorem{remark}[theorem]{Remark}

\newtheorem*{remark*}{Remark}
\newtheorem*{remarks*}{Remarks}

\def\@empty{}

\def\author#1{\par
    {\centering{\authorfont#1}\par\vspace*{0.05in}}}

\def\titlefont{\fontsize{12}{15} \centering{}}
\def\authorfont{\fontsize{12}{15}}

\let\affiliationfont\rhfont

\def\address#1{\par
    {\centering{\affiliationfont#1\par}}\par\vspace*{12pt}
}

\def\body{
\setcounter{footnote}{0}
\def\thefootnote{\alph{footnote}}
\def\@makefnmark{{$^{\rm \@thefnmark}$}}
}

\def\title#1{ 
    \thispagestyle{plain}
    \vspace{-30pt}
    \vskip 79pt
    {\centering{\titlefont #1\par}}%
    \vskip 1em
}

\providecommand{\keywords}[1]{\vspace{.3in} \textbf{Keywords:} #1}

\begin{document}
\title{\textbf{Maxima of a Random Model of the Riemann Zeta Function \\
over Intervals of Varying Length}}
\vspace{1.2cm}
\noindent\begin{minipage}[b]{0.37\textwidth}
\author{Louis-Pierre Arguin}
\address{CUNY, Baruch College \& Graduate Center\\
  louis-pierre.arguin@baruch.cuny.edu}
 \end{minipage}
\begin{minipage}[b]{0.3\textwidth}
\author{Guillaume Dubach}
\address{École Normale Supérieure -- PSL \\
guillaume.dubach@ens.fr}
\end{minipage}
\begin{minipage}[b]{0.3\textwidth}
\author{Lisa Hartung}
\address{JGU Mainz\\
lhartung@uni-mainz.de}
\end{minipage}

\date{March 2021}

\abstract{
We consider a model of the Riemann zeta function on the critical axis and study its maximum over intervals of length $(\log T)^{\theta}$, where $\theta$ is either fixed or tends to zero at a suitable rate. It is shown that the deterministic level of the maximum interpolates smoothly between the ones of log-correlated variables and of i.i.d.~random variables, exhibiting a smooth transition `from $\frac34$ to $\frac14$' in the second order. This provides a natural context where extreme value statistics of log-correlated variables with {\it time-dependent variance and rate} occur. A key ingredient of the proof is a precise upper tail tightness estimate for the maximum of the model on intervals of size one, that includes a Gaussian correction. This correction is expected to be present for the Riemann zeta function and pertains to the question of the correct order of the maximum of the zeta function in large intervals.
}


\section{Introduction}
The Riemann zeta function is defined for $\re s >1$ as
\begin{equation}\label{actual_zeta}
\zeta(s) = \sum_{n \geq 1} \frac{1}{n^s} = \prod_{p \geq 1} \frac{1}{1-p^{-s}},
\end{equation}
where the Euler product on the left is over all primes.
It extends uniquely to an analytic function over $\C\backslash\{1\}$. 
The Riemann hypothesis states that, except for the negative even integers (the so-called trivial zeros), the zeros of the function lie on the critical line $\re  s=1/2$.
The large values of the function on the critical line also plays a major role in number theory. The Riemann hypothesis implies that the maximum over a large interval $[T,2T]$ is
\begin{equation}
\label{eqn: RH}
\max_{t\in [T,2T]}|\zeta(1/2+\ii t)|=\OO\left(\exp \left(C\frac{ \log T}{\log\log T}\right)\right),
\end{equation}
for some $C>0$, as first shown in \cite{Lit24}. The question of whether or not the right-hand side is the true order of the maximum is still subject of debates, see for example \cite{FarGonHug07} for probabilistic arguments motivating
\begin{equation}
\label{eqn: FGH}
\max_{t\in [T,2T]}|\zeta(1/2+\ii t)|\approx \exp(\frac{1}{2}\sqrt{\log T\cdot\log\log T}).
\end{equation}

The question of the order of the maximum in shorter intervals is much better understood. Fyodorov, Hiary \& Keating \cite{FyoHiaKea12} and Fyodorov \& Keating \cite{FyoKea14}
conjectured using techniques of random matrix theory and log-correlated processes that, if $\tau$ is chosen uniformly on $[T,2T]$,
\begin{equation}
\label{eqn: FHK}
\max_{|h|\leq 1} |\zeta(1/2+\ii (\tau+h))|=\frac{\log T}{(\log\log T)^{3/4}}\ e^{\mathcal M(T)},
\end{equation}
where $(\mathcal M(T), T>1)$ is a tight sequence of random variables converging as $T\to \infty $ to a random variable $\mathcal M$ with right tail $\PP(\mathcal M>y)\sim C y e^{-2y}$.
The leading order $\log T$ was proved in \cite{Naj18} (on the Riemann hypothesis) and in \cite{ArgBelBouRadSou19} unconditionally. 
An upper bound for the leading order, and subleading order was proved in \cite{Har19} with an error of $\log\log \log T$. The upper tail tightness with the correct decay $y e^{-2y}$
was recently established in \cite{ArgBouRad20}.

The techniques developed so far can be also used to investigate large values in intervals with length varying with $T$.
It was conjectured in \cite{ArgOuiRad19} that, for intervals of size $(\log T)^{\theta}$ where $\theta>0$ is fixed, the maximum is
\begin{equation}
\label{eqn: FHK extend}
\max_{|h|\leq (\log T)^\theta} |\zeta(1/2+\ii (\tau+h))|=\frac{(\log T)^{\sqrt{1+\theta}}}{(\log\log T)^{\frac{1}{4\sqrt{1+\theta}}}}\ e^{\mathcal M_\theta(T)},
\end{equation}
where $(\mathcal M_\theta(T), T>1)$ is a tight sequence of random variables.
The leading order $(\log T)^{\sqrt{1+\theta}}$ was also proved there. 
This would suggest an interesting jump discontinuity in the subleading order in \eqref{eqn: FHK} and \eqref{eqn: FHK extend} where the exponent $\frac{1}{4\sqrt{1+\theta}}$ does not approach $\frac{3}{4}$ as $\theta\to 0$. The main goal of this paper is to shed light on this discontinuity.

Using a random model of the Riemann zeta function, we motivate the following conjecture that identifies the proper scale to smoothe the discontinuity:
\begin{conj}
Let $0<\alpha<1$ and $\theta\sim (\log\log T)^{-\alpha}$. Consider an interval of size $(\log T)^{\theta}=\exp(\log \log T)^{1-\alpha}$. If $\tau$ is chosen uniformly on $[T,2T]$, we have
\begin{equation}
\label{eqn: conj}
\max_{|h|\leq \exp (\log\log T)^{1-\alpha}} |\zeta(1/2+\ii (\tau+h))|=\frac{(\log T)^{\sqrt{1+\theta}}}{(\log\log T)^{\frac{1+2\alpha}{4}}}\ e^{\mathcal M_\alpha(T)},
\end{equation}
for some tight sequence of random variables $(\mathcal M_\alpha(T), T>1)$.
In other words, the relation between Equations \eqref{eqn: FHK} for $\theta>0$ and \eqref{eqn: FHK extend} for $\theta=0$ is provided by taking $\theta\downarrow 0$ like $(\log\log T)^{-\alpha}$, $0<\alpha<1$.
\end{conj}
Note that the non-trivial exponent 
$$
\frac{1+2\alpha}{4}
$$ 
interpolates between the case $\theta>0$ corresponding to $\alpha=0$, and $\theta=0$ for $\alpha=1$.
This transition is proved for a random Euler product defined in the next section, see Theorems \ref{thm:thetafixed} and \ref{thm:thetamove}. 
Along the way, a correction for the decay $y e^{-2y}$ for the right tail of the maximum will be proved in the case $\theta=0$, see Theorem \ref{thm: right tail}.
This is relevant for making sense of the two competing scenarios \eqref{eqn: RH} and \eqref{eqn: FGH} for the global maximum of the zeta function as explained in the next section.

There has been much interest recently in the study of log-correlated processes with {\it time-inhomogeneous variances and rates}, see for example \cites{FanZei,Oui,BovHar14, BovHar15, MalMil, ArgOui, fels19, FelsHartung19, FelsHartung20,  Mal, bbcm18,ArgZin}. In particular, such models can be designed in such a way that extreme value statistics interpolate between those of log-correlated variables and those of i.i.d.~random variables \cites{KisSch, BovHar,belloum21}.
The above conjecture provides a natural context, the Riemann zeta function, where such interpolated extreme value statistics would occur. 
\subsection{Main Results}
The following random model has been proposed in \cite{Har13} to study the large values of $\log |\zeta|$ in a short interval $I$:
\begin{equation}\label{true_model}
W_T(h)= \sum_{p \leq T} \frac{\re (U_p p^{- i h})}{p^{1/2}}, \qquad h\in I,
\end{equation}
where the sum is over primes less than $T$ and where $(U_p, p \text{ primes})$ are i.i.d.~random variables distributed uniformly on the unit circle.
The model morally corresponds to a formal expansion of the logarithm of the Euler product in \eqref{actual_zeta} around $s=1/2+\ii(\tau+h)$. 
The identification of the random phases $p^{-\ii\tau}\leftrightarrow U_p$ is then made with the extra assumption that the $U_p$'s are independent, which is not exactly the case for the $p^{-\ii \tau}$'s. The resulting model \eqref{true_model}, when exponentiated, can thus be seen as a random Euler product.\footnote{Another possible approach to model the large values of the zeta function would be to randomize the Dirichlet sum in \eqref{actual_zeta} instead of the product. 
This model is investigated in \cite{AymoneHeapZhao}.}

It was shown in \cite{ArguinBeliusHarper} that with high probability
\begin{equation}
\label{eqn: ABH}
\lim_{T\to\infty}\ \max_{|h|\leq 1} \ \frac{W_T(h)-\log\log T}{\log\log\log T}=-\frac{3}{4} \text{ in probability,}
\end{equation}
thereby providing evidence for the conjecture \eqref{eqn: FHK} for the log of the actual zeta function.

This work is concerned with a Gaussian version of this model where the $U_p$'s are replaced by standard complex Gaussian
\begin{equation}\label{gaussian_model}
X_T (h) = \sum_{p \leq T} \frac{\re (G_p p^{- i h})}{p^{1/2}}, \quad h\in I,
\end{equation}
where the $G_p$'s are i.i.d.~standard complex Gaussian variables. As explained in Section \ref{sect: MGF}, the process $(X_T(h), |h|\leq 1)$ is a Gaussian log-correlated process with variance $\frac{1}{2}\sum_{p\leq T}p^{-1}=\frac{1}{2}\log\log T+\OO(1)$.
 
We choose to work directly with the Gaussian model \eqref{gaussian_model} to highlight the new ideas of the proof. 
A standard approach to studying model \eqref{true_model} is to resort to Gaussian comparison; so that working with a Gaussian process from the start makes the outline of the proof much clearer. This is done at virtually no cost to the strength of the result, as the two processes are known to be close, as far as the relevant statistics are concerned, see for example \cite{ArguinBeliusHarper} for a pointwise comparison relying on Berry-Esseen bounds, and Theorem 1.7 in \cite{SaksmanWebb} for a global approximation. 
It was shown in \cite{SaksmanWebb} that the exponential of the process $(W_T(h), |h|\leq 1)$, suitably normalized, converges to Gaussian multiplicative chaos. 
Any of these methods could be similarly applied on top of the present work, in order to extend its results to other models such as (\ref{true_model}).

The first result is concerned with intervals of length $(\log T)^{\theta}$ for fixed $\theta>0$, where i.i.d.~extreme value statistics prevail.
\begin{theorem}\label{thm:thetafixed} 
For $\theta>0$ fixed, we have 
\begin{equation}
\label{eqn: thetafixed}
\lim_{T\to\infty}\ \max_{|h|\leq (\log T)^\theta} \frac{X_T(h)-\sqrt{1+\theta}\ \log\log T}{\log\log \log T}=\frac{-1}{4\sqrt{1+\theta}} \text{ in probability.}	
\end{equation}
\end{theorem}
This is consistent with Equation \eqref{eqn: FHK extend}, after exponentiation. The result is easier to prove than Equation \eqref{eqn: ABH}, as correlations play a lesser role. 
The second and main result is an explicit interpolation between Equations \eqref{eqn: ABH} and \eqref{eqn: thetafixed}, when $\theta$ goes to zero with $T$ at a suitable rate.
\begin{theorem}\label{thm:thetamove} 
For $\al \in (0,1)$ and $\theta = (\log\log T)^{-\al}$, we have
\begin{equation}
\lim_{T\to\infty}\ \max_{|h|\leq (\log T)^\theta} \frac{X_T(h)-\sqrt{1+\theta}\ \log\log T}{ \log\log \log T}=\frac{-(1+2\alpha)}{4} \text{ in probability.}
\end{equation}
\end{theorem}
This is the basis for the conjecture \eqref{eqn: conj}. 
This result is also relevant for the pseudomoments of the Riemann zeta function where it provides evidence for the asymptotic behavior of such moments in certain regimes, see Conjecture 5.4 in \cite{Ger20}.
To our knowledge, these interpolated statistics are also new in the context of many weakly correlated branching random walks. This connection is explained in Section \ref{sect: structure}.

A key ingredient of the proof, which is of independent interest, is a precise upper bound for the right tail of the maximum in an interval of order one ($\theta=0$).
It gives an alternative and direct approach to proving the upper tail tightness for log-correlated processes, in particular for the 2D Gaussian free field, that avoids the use of Gaussian comparison \cites{Bis, BraDinZei}. 
\begin{theorem}
\label{thm: right tail}
Let $y>0$ and  $y=\oo\big(\tfrac{\log\log T}{\log\log\log T}\big)$. Then we have
$$
\PP\left(\max_{|h|\leq 1} X_T(h) > \log\log T -\frac{3}{4}\log\log\log T +y\right)\leq C y e^{-2y} e^{-y^2/\log\log T},
$$
for some constant $C>0$.
\end{theorem}
We expect that the result holds up to $y=\oo(\log\log T)$ by analogy with branching Brownian motion, see for example Proposition 2.1 in \cite{BovHar}. 
The presence of both the exponential and Gaussian decays is relevant to the question of the maximum of the zeta function in large intervals quantified by Equation \eqref{eqn: RH} and \eqref{eqn: FGH}, at least at the heuristic level. Indeed, if the Gaussian correction were not present, a back-of-the-envelope calculation for the maximum of independent maxima in intervals of order one shows that Equation \eqref{eqn: RH} should yield the true order of the maximum. However, Theorem \ref{thm: right tail} gives evidence that the Gaussian behavior prevails for large $y$, in which case Equation \eqref{eqn: FGH} should be closer to the truth.

\subsection{Structure of the Proof}
\label{sect: structure}
Conjecture \eqref{eqn: FHK} is based on the assumption that the statistics of the zeta function on an interval of size one of the critical line resemble the ones of the characteristic polynomial of a random unitary matrix (CUE). It has been known since Bourgade \cite{Bou} that the finite-dimensional distributions of the logarithm of $\zeta$ and of the characteristic polynomial of CUE are Gaussian, when suitably normalized, with correlations that decay logarithmically with the distance. 
In the context of the model $(X_T(h), h\in I)$, this is also verified as we have
\begin{equation}
\label{eqn: covariance}
\E[X_T(h)X_T(h')]=
\begin{cases}
\frac{1}{2}\log |h-h'|^{-1} +\OO(1) & \ \text{if $|h-h'|\leq 1$},\\
\OO(|h-h'|^{-1}) & \ \text{if $|h-h'|> 1$}.
\end{cases}
\end{equation}
This is explained in Section \ref{sect: MGF}. The dichotomy here is important: for intervals of size one, the variables $X_T(h), X_T(h')$ are log-correlated, whereas for large intervals, most pairs are weakly correlated.
Therefore, the extreme value statistics of the process $(X_T(h),|h|\leq (\log T)^{\theta})$ should interpolate between log-correlated and IID for $\theta\geq 0$.
We are mainly concerned with the intermediate regime where hybrid statistics should appear. 
Note that even though the process $(X_T(h), |h|\leq 1)$ is continuous, the analysis can be reduced to a finite number of points: this is made precise in the proof of Lemma \ref{lem: discretization}. The distance $(\log T)^{-1}$ between these points is consistent with Equation \eqref{eqn: covariance}, as in the case $|h-h'|\approx (\log T)^{-1}$ the covariance is essentially equal to the variance $1/2\log\log T$.

The following analogous problem is useful to keep in mind throughout the proof.
An important example of log-correlated process is branching Brownian motion (BBM), constructed as follows. 
At time $0$, there is a single Brownian motion at $0$. It evolves in time with diffusion constant $1/2$ so the variance at time $t$ is $t/2$. 
We choose $1/2$ to match the variance of $X_T(h)$.
After an exponentially-distributed time, the Brownian motion splits into two independent Brownian motions, with the same branching property. It is not hard to check that after a time $t$ this procedure produces an average of $e^t$ Brownian motions. It is known since Bramson \cite{Bram} that the maximum at time $t$ over all Brownian particles is
\begin{equation}
\label{eqn: BBM}
t-\frac{3}{4}\log t +\mathcal M_t,
\end{equation}
where the fluctuations $\mathcal M_t$ are known in the limit $t\to\infty$. 
In view of Equation \eqref{eqn: FHK}, we see that the extreme statistics of a BBM at time $t$ should be the same as the ones of $(X_T(h),|h|\leq 1)$, after the identification $\log\log T\longleftrightarrow t$. 

What about longer intervals?
It is not hard to check that the maximum of two independent BBM's (or any finite number for that matter) also satisfy \eqref{eqn: BBM}.
 Indeed, this is the same as the situation of a single BBM right after the first splitting, which occurs in a finite time.
Hence, the statistics of the maximum are not affected for $t$ large.
However, if one takes many (depending on $t$) independent BBM's, Equation \eqref{eqn: BBM} should eventually morph to IID statistics. 
Equation \eqref{eqn: covariance} suggests that the process $(X_T(h), |h|\leq (\log T)^\theta)$ should behave like $(\log T)^\theta$ independent copies of $(X_T(h), |h|\leq 1)$, because of the strong decoupling for $h$'s at a distance more than $1$. 
In particular, the right number of independent BBM's needed to match to the subleading order the extreme value statistics of $(X_T(h), |h|\leq (\log T)^\theta)$ would be $e^{t\theta}$.
To sum up, the proofs of Theorems \ref{thm:thetafixed} and \ref{thm:thetamove} are guided by the analysis of the extreme values of $e^{t\theta}$ independent BBM's.
We stress that the precise decay $\theta\sim t^{-\alpha}$ needed to capture the intermediate regime is a non-trivial consequence of Theorem \ref{thm: right tail}. 

The upper bounds of the theorems are derived in Section \ref{sect: UB}.
For $\theta>0$, the proof is not hard and is done in Section \ref{sect: UB thetafixed}. It follows by a simple union bound on the number of points, which is $(\log T)^{\theta}\cdot \log T=(\log T)^{1+\theta}$, since there are $e^t=\log T$ points in intervals of size one, and there are essentially $e^{t\theta}=(\log T)^\theta$ such intervals. 
The case $\theta\downarrow 0$ is much more subtle. It is established by taking a union bound over $e^{t\theta}$ intervals of size one, each of which have a precise right tail given in Theorem \ref{thm: right tail}. The calculation involves every factor in the estimate of the right tail. This is done in Section \ref{sect: UB thetamove}. 
The proof of the right tail of the maximum for $\theta=0$ is given in Section \ref{sect: right tail}.

The proof of the lower bound in Section \ref{sect: LB} is based on the multiscale refinement of the second moment method given in \cite{Kis}.
In the case $\theta=0$ analyzed in \cite{ArguinBeliusHarper}, it is necessary to truncate the second moment by introducing a linear barrier that prevents the partial sums of $X_T(h)$ to be too high, an idea going back to \cite{Bram} for a single BBM.
The linear barrier can be justified as follows: since there are on average $e^k$ Brownian motions at time $k$, a Gaussian estimate yields that the maximum at time $k$ of a single BBM should be $\approx \sqrt{k\cdot \log e^k}=k$.
There are some important variations in the intermediate regime where $\theta\sim t^{-\alpha}$.
Following the above heuristics for $e^{t\theta}$ independent BBM's, the maximum at time $0\leq k\leq t$ of $e^{t\theta}$ independent BBM's is $\sqrt{k\cdot \log e^{k+t\theta}}=\sqrt{k(k+t\theta)}$. In particular, this is much larger than $k$ for small $k$'s!
This means that the barrier should not be too small at the beginning to capture the typical path of the maximizer. 
Since the barrier is fairly high at the beginning, it does not affect the trajectories of a single path very much. 
As a matter of fact, it will be clear from the proof that the barrier is only effective for times larger than $t-t^{\alpha}$. 

It would be interesting to describe the finer asymptotics of the extreme values of $( X_T(h),|h|\leq (\log T)^{\theta})$ down to the fluctuations, a level of precision still elusive for the Riemann zeta function. 
 It seems feasible to prove tightness of the recentered maximum by carrying a finer analysis of the random walk avoiding the specific barrier used for the lower bound.
However, information on the specifics of the fluctuations seem much harder to obtain as one expects non-trivial arithmetic and random-matrix-type corrections of order one whenever $\theta>0$.
 Numerical evidence of this phenomenon appeared in \cite{AmzArgBaiHuRao}.\\

\noindent {\bf Notations.}
We write $f(T)=\OO(g(T))$ whenever $\limsup_{T\to\infty} |f(T)/g(T)| <\infty$ with possibly other parameters fixed, such as $\e$ and $\theta$ which will be clear from context.
We often use Vinogradov's notation $f(T)\ll g(T)$ whenever $f(T)=\OO(g(T))$.
 We also use $f(T)=\oo(g(T))$ whenever $\lim_{T\to\infty} |f(T)/g(T)|=0.$
Some expressions in the proofs are more palatable in a $\log\log$ scale. The parameters at that scale will be denoted by the corresponding lower case letter. 
For example, we write respectively 
\begin{equation}
\label{eqn: loglog}
 k=\log\log K\quad \ell=\log\log L\quad t=\log\log T,\quad  \text{for $K,L,T$}.
\end{equation}
Finally, we use the double-bracket notation $[\![k,t]\!]$ for $[k,t]\cap \Z$, the integers in the interval $[k,t]$.\\

\noindent {\bf Acknowledgements.}
The authors would like to thank the referee for the numerous insightful comments that led to a substantial improvement of the first version of the paper. 
The research of L.-P.~A. is supported in part by the granst NSF CAREER~DMS-1653602 and NSF~DMS-2153803. G.~D. gratefully acknowledges support from the European Union's Horizon 2020 research and innovation programme under the Marie Sk{\l}odowska-Curie Grant Agreement No. 754411.
The research of L.~H. is supported in part by the Deutsche Forschungsgemeinschaft (DFG, German Research Foundation)  through Project-ID 233630050 -TRR 146, Project-ID 443891315  within SPP 2265 and Project-ID 446173099.

\section{Probability Estimates for the Model}
\label{sect: MGF}
This preliminary section is devoted to giving more details about the correlation structure of the process $X_T$. 
It is similar to Section 2 of \cite{ArguinBeliusHarper}. The results here are more straightforward, as the process is exactly Gaussian.
We include them for completeness.

In some cases, we need to study the process in a smaller range of primes. For this purpose, we define, for $1\leq K < L\leq T$,
\begin{equation}
\label{eqn: XKL}
X_{K,L}(h)= \sum_{K<p \leq L} \frac{\re (G_p p^{- i h})}{p^{1/2}}, \quad |h|\leq (\log T)^\theta.
\end{equation}
We simply write $X_K(h)$ for $X_{1,K}(h)$. 
The Laplace transform of the process is easy to compute:
\begin{lemma}\label{lem: MGF}
Let $1\leq K < L\leq T$. For $\lambda, \lambda' \in \R$ and $|h|,|h'|\in \R$, we have
\begin{equation}\label{char_fun_bias}
\begin{aligned}
&\E[e^{\lambda X_{K,L}(h)+\lambda' X_{K,L}(h')}]=
\exp\Big(\sum_{K< p \leq L} \frac{1}{4p} \left( \lambda^2 + \lambda'^2 + 2 \lambda \lambda'\cos(|h-h'|\log p) \right) \Big).
\end{aligned}
\end{equation}
\end{lemma}

\begin{proof}
This is obvious from the fact that
\begin{equation}
|\lambda p^{-ih} + \lambda' p^{-ih'}|^2 = \lambda^2 + \lambda'^2 + 2 \lambda \lambda'\cos(|h-h'|\log p).
\end{equation}
Therefore the Gaussian random variable $ \lambda X_{K,L}(h) + \lambda' X_{K,L}(h') $ has variance
\begin{equation}
\sum_{K< p \leq L} \frac{1}{4p} \left( \lambda^2 + \lambda'^2 + 2 \lambda \lambda'\cos(|h-h'|\log p) \right)
\end{equation}
and the claim follows.
\end{proof}
Setting $\lambda'=0$ in the above yields the variance of $X_{K,L}(h)$
\begin{equation}
\label{eqn: mertens}
\frac{1}{2}\sum_{K<p\leq L}\frac{1}{p}=\frac{1}{2}(\ell-k)+\OO(e^{-c\sqrt{e^k}}), \quad K>2,
\end{equation}
where we use the $\log\log$-notation \eqref{eqn: loglog}.
The equality is a strong form of Mertens's theorem, and can be proved using the Prime Number Theorem with classical error, see for example \cite[Theorem 6.9]{MonVau2003}, 
\begin{equation}\label{eq: PNT}
 \#\{p \le x: p \mbox{ prime}\} = \int_2^{x} \frac{1}{\log u}\rd u + \OO( x e^{-c \sqrt{ \log x }} ).
\end{equation}
It is not hard to evaluate the same way the cosine sum appearing in the covariance of $X_{K,L}(h)$ and $X_{K,L}(h')$.
\begin{lemma}
\label{lem: cos sum}
Let $1\leq K < L\leq T$. We have, using the notation \eqref{eqn: loglog},
$$
\sum_{K< p \leq L} \frac{1}{p} \cos(\delta \log p) =
\begin{cases}
\ell-k+\OO(\delta^2e^{2\ell})+\OO(e^{-c\sqrt{e^k}})& \text{if $\delta e^\ell\leq 1$,}\\
\OO( \delta^{-1}e^{-k}) +\OO(e^{-c\sqrt{e^k}})&\text{if $\delta e^k> 1$}.
\end{cases}
$$
\end{lemma}
\begin{proof}
This is done in Lemma 2.1 of \cite{ArguinBeliusHarper}. We reproduce the details for completeness.
The sum can be expressed using integration by parts and Equation \eqref{eq: PNT} as
$$
 \int_{K}^{L} \frac{\cos(\delta \log u )}{u \log u} \rd u + \OO( e^{-c\sqrt{\log K}} ).
$$
The claim for $\delta e^\ell\leq 1$ follows from $\cos(\delta \log u) = 1 + \OO(\delta^2(\log u)^2)$.
When $\delta e^k> 1$, we do the change of variable $v=\log u$ and integrate by parts. The integral then becomes
$$
\frac{\sin(\delta v)}{\delta v}\Big|_{\log K}^{\log L}
+
\int_{\log K}^{\log L} \frac{\sin( \delta v )}{ \delta v^2} \rd u\ .
$$
Both terms are $\OO(\delta^{-1} e^{-k})$.
\end{proof}

Based on Lemma \ref{lem: cos sum}, we see that there are two very different regimes for the covariance based on the value of $|h-h'|$: for $1\leq K<L\leq T$,
\begin{enumerate}
\item If $|h-h'|<e^{-\ell}$, we get
\begin{equation}
\E \big[ X_{K,L}(h) X_{K,L}(h') \big] = \frac12 (\ell-k) + \OO(e^{2\ell}|h-h'|^{2}).
\end{equation}
In this case the two variables are almost perfectly correlated.

\item In the case $|h-h'|> e^{-k}$, the oscillatory behavior of the sum causes the variables to decorrelate polynomially
\begin{equation}
\label{eqn: cov decouple}
\E \big[ X_{K,L}(h) X_{K,L}(h') \big] = \OO(e^{-k}|h-h'|^{-1}).
\end{equation}
\end{enumerate}
It is useful to think of $X_{K,L}(h)$ as the sum of increments of a random walk between time $k$ and $\ell$. 
With this in mind, the above dichotomy says that the random walks at $h$ and $h$' are almost the same until time $-\log |h-h'|$, at which point they decorrelate rapidly. 

We make the decorrelation of the random walks more precise in the following lemma. 
The decorrelation will be needed under a tilted measure $\wPP$: for $\lambda,\lambda'\geq 0$, take
\begin{equation}
\label{eqn: two-point biased measure}
\frac{\rd \wPP}{\rd \PP}=\frac{e^{\lambda X_{K,L}(h)+\lambda'X_{K,L}(h')}}{\E[e^{\lambda X_{K,L}(h)+\lambda' X_{K,L}(h')}]}, \quad h,h'\in \R.
\end{equation}
It is straightforward to check by differentiating the logarithm of Equation \eqref{char_fun_bias} that the covariance is unchanged under $\wPP$, and that the mean is
\begin{equation}
\label{eqn: tilted mean}
\wE[X_{K,L}(h)]= \sum_{K< p \leq L} \frac{\lambda +\lambda' \cos(|h-h'| \log p)}{2p}.
\end{equation}

\begin{lemma}
\label{lem: decoupling}
Let $1\leq K < T$. Consider $A_{\ell}, A'_{\ell}$, a collection of intervals in $\R$, indexed by $\ell\in [\![k+1,t]\!]$. 
Then for any $h,h'$ with $|h-h'|>e^{-k}$, we have under the measure $\wPP$ defined by Equation \eqref{eqn: two-point biased measure}:
\begin{multline}\label{eqn:decoupling1}
\wPP(X_{K,L}(h)\in A_\ell, X_{K,L}(h')\in A'_\ell, \ \forall \ell \in [\![k+1,t]\!] )= \\
(1+ \OO((e^{k}|h-h'|)^{-\frac{1}{2}})) \ \wPP(X_{K,L}(h)\in A_\ell, \ \forall \ell ) \ \wPP(X_{K,L}(h')\in A'_\ell \ \forall \ell  ) \\
 +\OO(e^{-c\sqrt{e^{k} |h- h'|}}),
\end{multline}
where the error does not depend on the choice of $A$'s and $c$ denotes an absolute positive constant.
\end{lemma}

\begin{proof}
As the process $X$ has independent increments, the probability can be decomposed in terms of the variables
\begin{equation}
\label{eqn: Y}
Y_\ell(h)=\sum_{e^{\ell-1} < \log p \leq e^{\ell} } \frac{\re (G_p p^{-ih}) }{p^{1/2}}.
\end{equation}
This is $X_{L^{1/e},L}(h)$ in the notation of Equation \eqref{eqn: tilted mean}.
We prove that for every $\ell\in [\![k+1,t]\!]$
\begin{equation}\label{step_decoup_estimate}
\begin{aligned}
&\wPP \left( Y_\ell(h) \in B_\ell, \ Y_\ell(h')\in B_\ell' \right)=\\
&(1+\OO((e^{\ell}|h-h'| )^{-1/2}) \wPP \left( Y_\ell(h) \in B_\ell \right)
\wPP \left( Y_\ell(h')\in B_\ell' \right)
+ \OO(e^{-c\sqrt{e^{\ell} |h- h'|}}),
\end{aligned}
\end{equation}
for some intervals $B_\ell$ and $B_\ell'$,
where the error terms are uniform in $\ell$ and in the choice of $B$'s.
Equation \eqref{eqn:decoupling1} then follows by writing the probability increment by increment using a Markov-type decomposition and taking the product.

To prove Equation \eqref{step_decoup_estimate}, consider first the subset 
$$
E_\ell = \left\{  \| (Y_\ell(h),Y_\ell(h')) - (\mu_\ell,\mu_\ell) \|_{\R^2} \leq (e^{\ell}|h-h'|)^{1/4} \right\},
$$
where 
\begin{equation}
\label{eqn: mu_ell}
\mu_\ell=\wE[Y_\ell]= \sum_{e^{\ell-1}<\log p \leq e^\ell} \frac{\lambda +\lambda' \cos(|h-h'| \log p)}{2p}
\end{equation}
by Equation \eqref{eqn: tilted mean}. 
A straightforward Gaussian estimate (or a Chernoff bound using Lemma \ref{lem: MGF}) yields
$
\wPP (E_\ell^{\rm c} )
\leq
\exp \left( {-c\sqrt{e^{\ell}|h-h'|}} \right).
$
This accounts for the error term in Equation \eqref{step_decoup_estimate}. 

It remains to estimate the Gaussian density on the event $E_\ell$. The covariance matrix $\Sigma_\ell$ of  $(Y_\ell(h),Y_\ell(h'))$ is
\begin{equation}
\Sigma_\ell=
\left(
\begin{matrix}
\sigma_\ell^2 & \rho_\ell\\
\rho_\ell & \sigma_\ell^2
\end{matrix}
\right),
\quad \rho_\ell=\OO(e^{-\ell}|h-h'|^{-1}), \quad \sigma_\ell^2=\frac{1}{2}+\OO(e^{-c\sqrt{e^\ell}}),
\end{equation}
by Equation \eqref{eqn: cov decouple}.
The inverse covariance admits the following expansion around the identity.
\begin{equation}
\label{eqn: Sigma}
\Sigma_\ell^{-1}=\sigma_\ell^{-2}\left\{
 \left(
\begin{matrix}
1 & 0\\
0 & 1
\end{matrix}
\right)
- \e_\ell 
\left(
\begin{matrix}
0 & 1\\
1& 0
\end{matrix}
\right)
+ \OO(\e^2_\ell )
\left(
\begin{matrix}
1 & 1\\
1& 1
\end{matrix}
\right)
\right\},
\quad \e_\ell=\frac{\rho_\ell}{\sigma_\ell}=\OO(e^{-\ell}|h-h'|^{-1}).
\end{equation}
For $(y,y')\in E_\ell$, the above implies the following estimate:
$$
(y-\mu_\ell,y'-\mu_\ell)^T \Sigma_\ell^{-1}(y-\mu_\ell,y'-\mu_\ell)
=
\frac{(y-\mu_\ell)^2+(y'-\mu_\ell)^2}{\sigma_\ell^2}+\OO((e^{\ell}|h-h'|)^{-1/2}).
$$
The estimate for the determinant of $\Sigma^{-1}$ is immediate from Equation \eqref{eqn: Sigma}.
Therefore, the density of the Gaussian vector $(Y_\ell(h),Y_\ell(h'))$ under $\wPP$ is
$$
\big(1+ \OO((e^{\ell}|h-h'|)^{-1/2})\big)
\frac{1}{ 2\pi \sigma^2 } \exp\Big(-\frac{(y-\mu_\ell)^2+(y'-\mu_\ell)^2}{2\sigma_\ell^2}\Big).
$$
This proves Equation \eqref{step_decoup_estimate} and concludes the proof of the lemma.
\end{proof}

The above proof can be also used to compare the probability to the one of a random walk with increments of variance exactly equal to $1/2$.
We will only need this in the case of one point $h$. 
To use the standard version of the ballot theorem, we will need to compare the random walk at $h$ with a random walk with increments of variance exactly equal to $1/2$. 
\begin{corollary}
\label{cor: var 1/2}
Let $1\leq K < T$. Consider $A_{\ell}$, a collection of intervals in $\R$, indexed by $\ell\in [\![k+1,t]\!]$. Consider $\mathcal Y_\ell$, $\ell\geq 1$, i.i.d.~Gaussian random variables of mean $\mu_\ell$ as in Equation \eqref{eqn: mu_ell}, and variance $1/2$ under $\wPP$. 
Then for any $h\in\R$ and $\ell>k$, we have under the measure $\wPP$ defined by Equation \eqref{eqn: two-point biased measure}:
\begin{equation}
\label{eqn:decoupling2}
\begin{aligned}
&\wPP(X_{K,L}(h)\in A_\ell, \forall \ell\in [\![k+1,t]\!])\\
&=\big(1+\OO(e^{-c\sqrt{e^k}})\big)
\wPP\Big(\sum_{j=k+1}^\ell \mathcal Y_j\in A_\ell, \forall \ell\in [\![k+1,t]\!]\Big) +\OO(e^{-ce^{2k}}).
\end{aligned}
\end{equation}
\end{corollary}
\begin{proof}
The proof follows the same line as the one of Lemma \ref{lem: decoupling}, setting $\lambda'=0$ in \eqref{eqn: two-point biased measure}. Consider the increments $Y_\ell$ and the event
$$
E_\ell =\left \{ |Y_\ell(h) - \mu_\ell | \leq  e^{\ell} \right\}.
$$
A Gaussian estimate shows that $\PP(E_\ell^c)\ll e^{-c e^{2\ell}}$. For $y\in E_\ell$, we have that
$$
\frac{1}{\sigma_\ell}e^{-(y-\mu_\ell)^2/2{\sigma_\ell^2}}=(1+\OO(e^{-c\sqrt{e^\ell}}))e^{-(y-\mu_\ell)^2}.
$$
Putting both estimates together yield for any interval $B\subset \R$,
$$
\wPP(Y_\ell(h)\in B)=(1+\OO(e^{-c\sqrt{e^\ell}})) \wPP(\mathcal Y_\ell\in B\cap E_\ell)+\OO(e^{-ce^{2\ell}}).
$$
The probability of $X_{K,L}$ is then obtained by conditioning successively on the $Y_\ell$'s and computing the product. 
\end{proof}
We remark that if the $A_\ell$'s do not grow too fast in such a way that the event $\{X_{K,L}(h)\in A_\ell, \forall \ell\in [\![k+1,t]\!]\}$
is included in the events $E_\ell$ for all $ \ell\in [\![k+1,t]\!]$, the additive error term is not present and we may conclude that 
\begin{equation}
\label{eqn:decoupling3}
\begin{aligned}
&\wPP(X_{K,L}(h)\in A_\ell, \forall \ell\in [\![k+1,t]\!])\\
&=\big(1+\OO(e^{-c\sqrt{e^k}})\big)
\wPP\Big(\sum_{j=k+1}^\ell \mathcal Y_j\in A_\ell, \forall \ell\in [\![k+1,t]\!]\Big).
\end{aligned}
\end{equation}

\section{Upper Bounds}
\label{sect: UB}

\subsection{Upper Tail Tightness for $\theta=0$}
\label{sect: right tail}
The purpose of this section is to prove Theorem \ref{thm: right tail}. 
The starting point is in the spirit of Lemma 2.4 in Bramson, Ding, and Zeitouni \cite{BDZ}  for branching random walks; however, their approach has to be generalized to account for the lack of perfect branching in the model of the zeta function.
Our method is applicable to other log-correlated processes that do not exhibit an exact branching structure. 

As in Equation \eqref{eqn: Y}, we define
\begin{equation}\label{Cauchy_strips}
Y_k(h) = \sum_{e^{k-1} < \log p \leq e^k} \frac{1}{\sqrt{p}} \re (p^{-ih} G_p), \quad k\leq t,
\end{equation}
where $t=\log\log T$.
Each $Y_k(h)$ is a Gaussian variable of variance roughly $1/2$. We also define the partial sums
\begin{equation}
\label{eqn: S}
S_j(h) = \sum_{k=1}^{j} Y_k (h),\quad j\leq t,
\end{equation}
and their recentered version,
\begin{equation}
\ovS_j = S_j - m(j).\, \qquad m(j) = j\Big(1 - \frac34 \frac{\log t}{t}\Big).
\end{equation}
The method requires to introduce a logarithmic barrier: 
\begin{equation}\label{log_barrier}
\psi_j = \log (\min (j, t-j)), \quad 1\leq j\leq t-1.
\end{equation}
so that we have the symmetry $\psi_j = \psi_{t-j}$. We also set $\psi_0=\psi_t=0$. 

The idea of the proof is as follows. 
The goal is to bound the probability $p_{\text{cross}}$ that the value at some $h$ is above the level $m(t) + y$:
\begin{equation}\label{def_pcross}
p_{\text{cross}}
=
\PP \left( \max_{|h|\leq 1} X_{T} (h) > m(t) + y \right).
\end{equation}
We start by bounding $p_{\text{cross}}$ by a more amenable expression.
The event can be decomposed according to the first time $j$ after some fixed time $r$ when the barrier $y+\psi$ is crossed by the maximum over $h$. 
It is convenient to take $r$ away from $1$ for two reasons.
First, the increments for small primes have a variance, which is not close to $1/2$, which makes the estimate such as Proposition \ref{thm_ballot} more cumbersome.
The choice $$r=r(y)=y\vee 1$$ is good. 
Second, it is then possible to get an {\it a priori bound} on $\max_{|h| \leq 1} |S_j (h)|$ for all $r<j\leq t$. 
Indeed, a union bound gives
$$
\begin{aligned}
\PP(\exists r<j\leq t: \max_{|h| \leq 1} |S_j (h)|> 2j +y)
&\ll \sum_{j=r+1}^t e^j \PP( |S_\ell (h)|> 2j +y)\\
&\ll \sum_{j=r+1}^t e^j\cdot e^{-(2j+y)^2/2\sigma_j^2}\\
&\ll e^{-3r}\cdot e^{-4y}e^{-y^2/t}.
\end{aligned}
$$
The first inequality is obtained by a discretization argument (see Corollary 2.6 in \cite{ArguinBeliusHarper} which is proved again in Lemma \ref{lem: discretization} below).
The third inequality comes from the fact that $\sigma_j^2=j/2+\OO(1)$ and $y$ is much smaller than $t$. 
The above is much smaller than the claimed bound in Theorem \ref{thm: right tail}.
Therefore, we can assume without loss of generality in the proof that
\begin{equation}
\label{eqn: a priori}
 \max_{|h| \leq 1} |S_j (h)|\leq 2j +y, \quad \text{for all $r<j\leq t$.}
\end{equation}

We now estimate Equation \eqref{def_pcross} with an union bound on $r<j\leq t$:
\begin{equation}\label{first_time}
p_{\text{cross}}
\leq
\sum_{r<j \leq t-1} \PP \left( \forall r<\ell \leq j \ \max_{|h| \leq 1} \ovS_\ell (h) \leq y + \psi_\ell, \
\max_{|h|\leq 1} \ovS_{j+1}(h) > y+ \psi_{j+1} \right).
\end{equation}
Secondly, the interval $[-1,1]$ can be partitioned into $e^j$ intervals of width $2e^{-j}$ and centers $h^{(i)}_j$ for $i\leq e^j$. 
Therefore, the probability can be bounded by a union bound over these intervals:
\begin{equation}
p_{\text{cross}}
\ll
\sum_{r<j \leq t-1} 
\sum_{i=1}^{e^j} 
\PP \left( \forall r<\ell \leq j \ \max_{|h|\leq 1} \ovS_\ell(h) \leq y + \psi_\ell, \
\max_{|h-h^{(i)}_j|\leq e^{-j}} \ovS_{j+1}(h) > y+ \psi_{j+1} \right).
\end{equation}
Finally, the constraint on the maximum over $[-1,1]$ is limited to $h=h_j$. 
Using the symmetry of the distribution of the model under translation, the above expression then becomes
\begin{equation}\label{bound_pcross}
p_{\text{cross}}
\ll
\sum_{r<j \leq t-1} 
e^j 
\PP \left( \forall r<\ell \leq j \ \ovS_\ell(0)  \leq y + \psi_\ell, \
\max_{|h|\leq e^{-j}} \ovS_{j+1}(h) > y+ \psi_{j+1} \right).
\end{equation}
We now make a simple but important observation:
the overshoot of $\max_{|h|\leq e^{-j}} \ovS_{j+1}(h)$ is either due to a large value of the increment $\ovS_{j+1}- \ovS_j$ at $h=0$ or to 
a large value of the difference $S_{j+1}(h)-S_{j+1}(0)$ for some $|h|\leq e^{-j}$.
The first case is exactly the one that is obtained for an exact branching random walk. 
The second case is due to the fact the correlation structure of the model is not exactly a tree.
However, these fluctuations can be controlled very precisely by a chaining argument because of the log-correlations.
Specifically, define the following event that includes the barrier and the value of the endpoint $S_j$:
\begin{equation}
\label{eqn: Bj def}
\mathscr{B}_{j}(x) = 
\left\{
\forall r< \ell \leq j, \ \ovS_\ell(0) \leq y + \psi_\ell
\quad \text{and} \quad
\ovS_j(0) \in I_x 
\right\}, \quad j\leq t, \ x\in \Z,
\end{equation}
where $I_x=(x,x+1]$ for $x>0$, $I_x=(x-1,x]$ if $x<0$, and $I_x=[-1,1]$ if $x=0$.
Performing a union bound on the range of values of $\ovS_j(0)$, we prove the following two estimates:
\begin{lemma}
\label{lem: center}
With the notation above, we have for $y>0$ and $y=\oo\big(\tfrac{\log\log T}{\log\log\log T}\big)$,
\begin{equation}
\label{eqn: to prove 1}
\sum_{r< j \leq t-1}
\sum_{x\leq y+\psi_j} e^j\P \left( \mathscr{B}_{j}(x), \ovS_{j+1}(0) - \ovS_j (0) > y+\psi_{j+1}-x \right)\ll(y+1) e^{-2y} e^{-y^2/t}.
\end{equation}
\end{lemma}
This is the same estimate one would prove for a branching random walk. 
The new input is an estimate for the overshoot due to a large value in the difference in a small neighborhood:
\begin{lemma}
\label{lem: fluct}
With the notation above, we have for $y>0$ and  $y=\oo\big(\tfrac{\log\log T}{\log\log\log T}\big)$,
\begin{multline}
\label{eqn: to prove 2}
\sum_{r<j \leq t-1}
\sum_{x\leq y+\psi_{j+1}} e^j\P \left( \mathscr{B}_{j+1}(x), \max_{|h|\leq e^{-j}}\ovS_{j+1}(h) - \ovS_{j+1} (0) > y+\psi_{j+1}-x \right) \\
\ll(y+1) e^{-2y} e^{-y^2/t}.
\end{multline}
\end{lemma}

\begin{proof}[Proof of Theorem \ref{thm: right tail}]
 From the previous discussion, it suffices to bound the right-hand side of Equation \eqref{bound_pcross}. 
As explained above, this splits into the two cases given by Equations \eqref{eqn: to prove 1} and \eqref{eqn: to prove 2}, thereby proving the claim.
\end{proof}

We now prove the two lemmas. 
The first proof is along the lines of Lemma 2.4 of \cite{BDZ}.
\begin{proof}[Proof of Lemma \ref{lem: center}]
Consider the change of measure, from $\P$ to $\ovP$, as in Equation \eqref{eqn: two-point biased measure} with $\lambda'=0$,
\begin{equation}\label{CoM1}
\frac{\dd \ovP}{\dd \PP}
= \frac{e^{\lambda S_j(0)}}{\E \left[e^{\lambda S_j(0) }\right]}
=
e^{- \lambda^2 \sigma_j^2 / 2 }
e^{\lambda S_j (0)},
\end{equation}
where $\sigma_j^2=\frac{1}{2}\sum_{\log p\leq e^j}1/p=\frac{1}{2} j +\OO(1)$ as in Equation \eqref{eqn: mertens}, and $\lambda$ is chosen in such a way that $\ovE[\ovS_j]=0$ as in Equation \eqref{eqn: tilted mean}. 
The condition for this centering to hold is $ m(j) = \lambda \sigma_j^2 = \frac12 \lambda j (1 + \OO(1/j))$, which yields the following identities:
\begin{align}[left=\empheqlbrace]
\lambda & = \frac{m(j)}{\sigma_j^2} = \left( 2 - \frac32 \frac{\log t}{t} \right)(1+ \OO(1/j))
= 2 - \frac32 \frac{\log t}{t} + \OO(1/j), \label{lambdazero_1} \\
\frac{\lambda^2 \sigma_j^2}{2} & = \left( j  - \frac32 \log t + \frac{9}{16} \frac{(\log t)^2}{t} \right) (1 + \OO(1/j))
= j  - \frac32 \frac{\log t}{t} j + \OO(1), \label{lambdazero_2} \\
\lambda m(j) & = \left( 2j - 3 \log t + \frac98 \frac{(\log t)^2}{t}\right) (1+ \OO(1/j))
= 2j - 3 \frac{\log t}{t} j + \OO(1). \label{lambdazero_3}
\end{align}
The probability in the sum can be bounded using $\ovP$ as follows, for $w=y+\psi_{j+1}-x$,

\begin{equation}
\label{eqn: Bj}
\begin{aligned}
\PP\left(
\mathscr{B}_{j}(x) ,
\ovS_{j+1}(0) - \ovS_{j}(0)> w
\right)
& \ll
 e^{j - \frac32 \frac{\log{t}}{t} j } \cdot 
\ovE \left[
e^{- \lambda S_j(0)}  
\mathbf{1}_{\mathscr{B}_{j}(x)}
\mathbf{1}_{\ovS_{j+1}(0) - \ovS_{j}(0)> w}
\right] \\
& \ll
 e^{-j + \frac32 \frac{\log{t}}{t} j } \cdot
\ovE \left[
e^{- \lambda \ovS_j(0)}
\mathbf{1}_{\mathscr{B}_{j}(x)}
\mathbf{1}_{\ovS_{j+1}(0) - \ovS_{j}(0)>w}
\right]\\
& \ll
e^{-j} e^{\frac32 \frac{\log{t}}{t} j}
e^{-\lambda x}\cdot 
\ovP \left(
\mathscr{B}_{j}(x)
\right)
e^{-202w}
\\
\end{aligned}
\end{equation}
where \eqref{lambdazero_2} is used in the first inequality, \eqref{lambdazero_3} in the second, as well as the values of $\ovS_{j}(0)$ on the event $\mathscr{B}_{j}(x)$ and a Chernoff bound on $\ovS_{j+1}(0) - \ovS_{j}(0)$ in the third.
We could have used a Gaussian estimate for $\ovS_{j+1}(0) - \ovS_{j}(0)$.
However, as it will be clear below, the convergence of the sum is ensured as long as the probability for the increment is $\ll e^{-(2+\delta) w}$
for some $\delta>0$. We choose $\delta=200$. This highlights the similarities with the proof of Lemma \ref{lem: fluct}, where Gaussian estimates are not available.

We split the estimate of Equation \eqref{eqn: Bj} into two cases.
The estimate for $r<j\leq t/\log t$ is simpler, as no ballot theorem is needed. 
We trivially have 
\begin{equation}
\label{eqn: small Bj}
\ovP \left(\mathscr{B}_{j}(x)\right)\ll e^{-x^2/2\sigma_j^2}\ll e^{-x^2/j},
\end{equation}
since $\sigma_j^2=\frac{1}{2} j +\OO(1)$ and $|x|\ll j$, $j>r$, by the {\it a priori} bound \eqref{eqn: a priori}.
Therefore, the following estimate holds
\begin{equation}
\label{eqn: small j}
\PP\left(
\mathscr{B}_{j}(x) ,
\ovS_{j+1}(0) - \ovS_{j}(0)> w
\right)
 \ll
 e^{-j}\cdot e^{-2x}\cdot e^{-x^2/j}\cdot e^{-202w}.
\end{equation}
The assumption $y=\oo(t/ \log t)$ is needed in estimating $\lambda x$, so that $e^{\frac{3}{2}\frac{\log t}{t} x}\leq e^{\frac{3}{2}\frac{\log t}{t} (y+\psi_j)}\ll 1$.
In the case $t/\log t < j<t$, we first note that $\tfrac{3}{2} \frac{\log{t}}{t} j\leq j^{3/2}$, since the function $\log x/x$ is decreasing for $x\geq 1$.
Moreover, Proposition \ref{thm_ballot} can be applied thanks to the  {\it a priori bounds} \eqref{eqn: a priori} and Equation \eqref{eqn:decoupling2}.
All this implies
\begin{equation}
\label{eqn: large j}
\PP\left(
\mathscr{B}_{j}(x) ,
\ovS_{j+1}(0) - \ovS_{j}(0)> w
\right)
 \ll
e^{-j}
e^{-2x}
(y+1) (y+ \psi_{j}-x+1) e^{-x^2/j} e^{-202w}.
\end{equation}
Note that the right-hand side is also an upper bound to \eqref{eqn: small j}.
Thus, the sum of the lemma becomes
$$
\begin{aligned}
&\sum_{r<j\leq t-1}\sum_{x\leq y+\psi_j} e^j\P \left( \mathscr{B}_{j}(x), \ovS_{j+1}(0) - \ovS_j (0) > y+\psi_{j+1}-x \right) \\
&\ll \sum_{r<j\leq t-1}\sum_{x\leq y+\psi_j}e^{-2x}
(y+1) (y+ \psi_{j}-x+1) e^{-x^2/j}e^{-202(y+\psi_{j+1}-x)}.
\end{aligned}
$$
It is convenient to do a change of variable $v=y+\psi_j-x$, expressing the distance between the barrier and $S_j(0)$.
With this notation and the fact that $|\psi_{j+1}-\psi_j|<1$, the above is
\begin{equation}
\label{eqn: both}
\begin{aligned}
&\ll(y+1) \sum_{r<j\leq t-1}\sum_{v\geq 0}e^{-2(y+\psi_j-v)}
(v+1) e^{-(y+\psi_j-v)^2/j}e^{-202(v-1)}\\
&\ll (y+1)e^{-2y} \sum_{j\leq t-1}e^{-2\psi_j}\sum_{v\geq 0} (v+1) e^{-(y+\psi_j-v)^2/t}e^{-200v},
\end{aligned}
\end{equation}
 where we bounded $e^{-(y+\psi_j-v)^2/j}$ by $e^{-(y+\psi_j-v)^2/t}$ and extended the sum to $j\leq r$.
This allows us to use the symmetry $\psi_j = \psi_{t-j}$ to pair the terms $j$ and $t-j$ together. Indeed, for a given $v$, $y+\psi_j-v$ has the same value for both values of $j$, but $(y+\psi_j-v)^2 /t$, so that it is enough to bound the second half of the sum:
\begin{equation}
(y+1)e^{-2y}
\sum_{ j = \lfloor t/2 \rfloor}^{t-1}
e^{-2\psi_j}\sum_{v\geq 0} (v+1)e^{-100v}\cdot e^{-(y+\psi_j-v)^2/t}e^{-100v}.
\end{equation}
 Note that the term $e^{-(y+\psi_j-v)^2/t}e^{-100v}$ is maximal at $v=y+\psi_j-50t$, which is negative since $y=\oo(t/\log t)$. In particular, this implies that $v=0$ is the maximizer in the range considered, hence, $e^{-(y+\psi_j-v)^2/t}e^{-100v}< e^{-(y+\psi_j)^2/t}\leq e^{-y^2/t}$ since $y> 0$. With this observation, we conclude that the above is
$$
\ll (y+1)e^{-2y}e^{-y^2/t}
\sum_{ j = \lfloor t/2 \rfloor}^{t-1} e^{-2\psi_j}\sum_{v\geq 0} (v+1)e^{-100v}\ll (y+1)e^{-2y}e^{-y^2/t}, 
$$
by the definition of $\psi_j$, which ensures summability. This proves the claim. 
\end{proof}

\begin{proof}[Proof of Lemma \ref{lem: fluct}] 

We now decompose according to the value of $\ovS_{j+1}(0)$. We need to bound:
\begin{equation}\label{bound2_pcross}
\sum_{r<j \leq t-1} \sum_{x\leq y+\psi_{j+1}} e^j 
\PP \left( \mathscr{B}_{j+1} (x), \
\max_{|h|\leq e^{-j}} \ovS_{j+1}(h)- \ovS_{j+1}(0) > y+ \psi_{j+1}-x \right).
\end{equation}
The term $x=y+\psi_{j+1}$ must be dealt with separately. This is to ensure that we consider a jump of size at least $1$. 
To bound this term, the restriction on the maximum is simply dropped. 
More precisely, proceeding exactly as in Lemma \ref{lem: center} with a change of measure and using the ballot theorem, but dropping the event on the increment $S_{j+1}(0)-S_j(0)$, one gets
$$
\sum_{r<j \leq t-1} 
e^j \PP \left( \mathscr{B}_{j+1} (y+\psi_{j+1})\right)\ll (y+1) \sum_{j\leq t}e^{-2(y+\psi_j)}e^{-(y+\psi_j)^2/j}\ll (y+1)e^{-2y}e^{-y^2/t},
$$
as desired.

Let $v=y+\psi_{j+1}-x$. We can now restrict the proof to the case $v\geq 1$. On the event considered, there is an $h \in [-e^{-j}, e^{-j}]$ such that
$$
\ovS_{j+1}(h) - \ovS_{j+1}(0) > v \geq 1.
$$
We decompose this increment using a chaining argument. 
For each $h$, we consider a sequence $h_k \rightarrow h$, within the following dyadic sets
\begin{equation}
h_k \in\mathscr{H}_k  \text{ where }
\mathscr{H}_k = 
\left\{ e^{-j} \frac{l}{2^k}, \
l \in [ -2^{k}, 2^{k} ]\cap\Z \right\}
\subset [- e^{-j}, e^{-j}], \ k\geq 0.
\end{equation}
In particular, we can choose it such that $h_0=0$, and $|h_{k+1} - h_k| \leq \frac{e^{-j}}{2^{k}}$ (by taking $h_k$ the closest to $h$ in $\mathscr{H}_k$ when $k>0$). By continuity, we have
$$
\ovS_{j+1} (h) - \ovS_{j+1}(0) = 
\sum_{k=0}^{\infty} 
\left( \ovS_{j+1}(h_{k+1}) - \ovS_{j+1}(h_k) \right).
$$
The left-hand side being larger than $v \geq 1$, there is a $k^\star\geq 0$ such that
$$
\ovS_{j+1}(h_{k^\star+1}) - \ovS_{j+1}(h_{k^\star}) 
> 
\frac{v}{2 (1+k)^2}.
$$

We perform a union bound over every possibility of neighboring elements of $\mathscr{H}_{k+1}$. Let $(h_i)_i$ be an increasing enumeration of all elements of $\mathscr{H}_{k+1}$.
With these observations, the contribution to Equation \eqref{bound2_pcross} coming from $v\geq 1$ is bounded by
\begin{equation}\label{bound3_pcross}
 \ll
\sum_{j \leq t-1}
\sum_{v \geq 1}
\sum_{k \geq 0}
e^{j} 
\sum_{h_i\in \mathscr{H}_{k+1}} 
\PP\left(
\mathscr{B}_{j+1}(y+\psi_{j+1}-v),
\ovS_{j}(h_{i+1}) - \ovS_{j} (h_{i}) > \frac{v}{2 (1+k)^2}
\right).
\end{equation}
We use Chernoff's inequality
with one real parameters $\lambda = \la_k$ (to be specified later), and obtain the bound
\begin{equation}\label{bound4_pcross}
\sum_{j \leq t}
\sum_{v \geq 1}
\sum_{k \geq 0}
    e^j 
   \sum_{h_i\in \mathscr{H}_{k+1}}
    e^{- \la \frac{v}{2 (1+k)^2}}
    \E \left[ \mathbf{1}_{\mathscr{B}_{j+1}(y+\psi_{j+1}-v)}
    e^{\la (\ovS_{j+1}(h_{i+1}) - \ovS_{j+1}(h_{i})) }
    \right].
\end{equation}

The key step is to quantify the effect of the exponential factor on the probability of the event $\mathscr{B}_{j+1}$.
For this, note that for each $h_i\in \mathscr{H}_{k+1}$,
$$
\ovS_{j+1}(h_{i+1}) - \ovS_{j+1}(h_{i}) 
= S_{j+1}(h_{i+1}) - S_{j+1}(h_{i})
= \sum_{\ell = 1}^{j+1} (Y_\ell(h_{i+1}) - Y_\ell(h_{i})).
$$
In order to make the correlation with the path at $h=0$ explicit, we write the orthogonal decomposition of these increments, i.e., we consider the orthogonal $L^2$-projection of these increments $Y_\ell(h_{i+1}) - Y_\ell(h_{i})$ on $Y_\ell(0)$:
\begin{equation}
\label{eqn: ortho}
Y_\ell(h_{i+1}) - Y_\ell(h_{i}) 
=
\pi^{(h_i)}_\ell Y_\ell(0) +
\Delta^{(h_i)}_\ell,
\end{equation}
where $\Delta^{(h_i)}_\ell$ is a Gaussian variable independent of $Y_\ell(0)$ \textit{and} of all the other increments $Y_m$, $m \neq \ell $. 
We now estimate $\Delta^{(h_i)}_\ell$ and $\pi^{(h_i)}_\ell$. The bounds we obtain are the same for all $h_i\in \mathscr{H}_{k+1}$.
We first compute the variance of the difference
\begin{eqnarray}
    \rm{Var} (Y_\ell(h_{i+1}) - Y_\ell(h_{i}))
 &&= \sum_{e^{\ell-1} < \log p \leq e^{\ell}} \frac{1}{p} \left( 1 - {\cos ( |h_{i+1}-h_{i}| \log p )} \right)\nonumber\\
&&\leq \sum_{e^{\ell-1} < \log p \leq e^{\ell}} \frac{(\log p)^2}{p} (h_{i+1}-h_{i})^2.
\end{eqnarray}
Note that $|h_{i+1}-h_{i}|\leq e^{-j} 2^{-k}$ for all $h_i\in \mathscr{H}_{k+1}$ by construction, and $\log p \leq e^{\ell}$ in this sum, so that
$$
\rm{Var}(Y_\ell(h_{i+1}) - Y_\ell(h_{i})) \leq \sum_{e^{\ell-1} < \log p \leq e^{\ell}} \frac{1}{p} (e^{\ell})^2 e^{-2j} 2^{-2k}
\ll 2^{-2k}  e^{2(\ell-j)} .
$$
The Cauchy-Schwarz inequality then implies
$$
\left| \E \left[ (Y_\ell(h_{i+1}) - Y_\ell(h_i))\ Y_\ell(0) \right] \right|
\leq \E \left[ (Y_\ell(h_{i+1}) - Y_\ell(h_i))^2 \right]^{1/2} \cdot\E [Y_\ell(0)^2]^{1/2} 
\ll 2^{-k} e^{-(j-\ell)} ,
$$
for all $h_i\in \mathscr{H}_{k+1}$.
The bound for $\pi^{(h_i)}_\ell$ is then straightforward
\begin{equation}
\label{eqn: pi}
\pi^{(h_i)}_\ell=\frac{ \E \left[ (Y_\ell(h_{i+1}) - Y_\ell(h_i)) \ Y_\ell(0) \right]}{\E [Y_\ell(0)^2] }\ll  2^{-k} e^{-(j-\ell)}.
\end{equation}
The orthogonality in \eqref{eqn: ortho} implies
$$
\rm{Var}\left(Y_\ell(h_{i+1}) - Y_\ell(h_i)\right) =\E\left[\left(\Delta^{(h_i)}_\ell\right)^2\right]+\left(\pi^{(h_i)}_\ell\right)^2\ \E\left[Y_\ell(0)^2\right]\ .
$$
Together with \eqref{eqn: pi}, this gives the bound:
\begin{equation}
\label{eqn: delta}
\E\left[\left(\Delta^{(h_i)}_\ell\right)^2\right] \ll  2^{-2k} e^{-2(j-\ell)}.
\end{equation}
We now insert the decomposition \eqref{eqn: ortho} into the bound \eqref{bound4_pcross}. 
Equation \eqref{eqn: delta} yields the following exponential bound, for some absolute $c>0$,
\begin{equation}\label{eqn: expdelta}
\E \left[ e^{ \la \sum_{\ell \leq j+1} \Delta^{(h_i)}_\ell} \right]
\ll e^{c \sum_{\ell  \leq j+1} \la^2 e^{-2(j-\ell)} 2^{-2k} }
=  e^{\OO \left( \la^2 2^{-2k} \right)},
\end{equation}
since $\ell\leq j+1$. Note that the bounds in \eqref{eqn: pi}, \eqref{eqn: delta}  and \eqref{eqn: expdelta} are the same for all $h_i\in \mathscr{H}_{k+1}$.
Provided that $\la=\la_k \ll 2^k$ (our choice of $\la$ in the end will meet this criterion), the bound (\ref{bound4_pcross}) becomes
\begin{equation}\label{bound5_pcross}
\ll\sum_{r<j \leq t-1}
\sum_{v \geq 1}
\sum_{k \geq 0}
    e^j 
    \sum_{h_i\in \mathscr{H}_{k+1}}
    e^{- \la \frac{v}{2 (1+k)^2}}
\E \left[ e^{\la \sum_{l \leq j+1} \pi^{(h_i)}_\ell Y_\ell(0)} \mathbf{1}_{\mathscr{B}_{j+1}(y+\psi_{j+1}-v)} \right].
\end{equation}
The exponential factor is now amenable to analysis, since its effect is to put a small drift on $Y_\ell(0)$.
More precisely, we combine the change of measure (\ref{CoM1}) from $\P$ to $\overline{\P}$ with a second change of measure from $\overline{\P}$ to $\widetilde{\P}$ such that
\begin{equation}\label{CoM2}
\frac{\dd \widetilde{\mathbb{P}}}{\dd \overline{\mathbb{P}}}
= \frac{e^{\la \sum_{l \leq j+1} \pi^{(h_i)}_\ell Y_\ell (0)}}{\overline{\E} \left[ e^{\la \sum_{l \leq j+1} \pi_\ell^{(h_i)} Y_\ell (0) }\right]}.
\end{equation}
Under $\widetilde{\mathbb{P}}$, the increments $Y_\ell(0)$, $\ell\leq j+1$, are independent Gaussians of variance $\E[Y_\ell(0)^2]\ll 1$ and shifted mean $\mu_\ell + \la \pi^{(h_i)}_\ell  \E[Y_\ell(0)^2]$. 
It is straightforward to check that the cumulative drift is of order $1$ for all $h_i\in \mathscr{H}_{k+1}$ using \eqref{eqn: pi}:
$$
\lambda \sum_{\ell\leq j+1} \pi^{(h_i)}_\ell
\leq
\lambda \sum_{\ell \leq j+1} e^{-(j-\ell)} 2^{-k}
\ll \lambda 2^{-k} \ll 1,
$$
again, provided that $\lambda \ll 2^k $.
By interpreting the small upward drift at each $\ell$ by a variation of the barrier, we have from Equation \eqref{eqn:decoupling2} and Proposition \ref{thm_ballot} the estimate
\begin{equation}
\label{eqn: Btilde}
\widetilde{\mathbb{P}} \left( \mathscr{B}_{j+1}(y+\psi_{j+1}-v) \right)
\ll
\frac{(y+1) (v+1)}{j^{3/2}} e^{-(y+\psi_{j+1}-v+c)^2 /j}, \quad j\geq t/\log t,
\end{equation}
by simply raising the original barrier by a constant $c$.
The endpoint value is also affected by this, but the constant $c$ will not matter in the end.
In the case, $j\leq t/\log t$, we can simply use that 
$$
\widetilde{\mathbb{P}} \left( \mathscr{B}_{j+1}(y+\psi_{j+1}-v) \right)
\ll e^{-(y+\psi_{j+1}-v+c)^2 /j}, \quad j\leq t/\log t,
$$
as in Equation \eqref{eqn: small Bj}. 
Therefore, bounding Equation \eqref{bound5_pcross} first by introducing $\ovP$ as in Equation \eqref{eqn: Bj}, then by introducing $\widetilde{\mathbb{P}}$ and using \eqref{eqn: Btilde},
we are left with the bound
\begin{equation}
\label{eqn: chaining almost}
(y+1)
\sum_{r<j \leq t-1}
\sum_{v \geq 1}
\sum_{k \geq 0}
2^k
e^{- \lambda \frac{v}{2 (1+k)^2}} 
e^{-2 (y+\psi_{j+1}-v)}
(v+1) e^{-(y+\psi_{j+1}-v+c)^2 /j},
\end{equation}
as $\mathscr{H}_{k+1}$ has of order $2^k$ elements.
The only constraint on $\la$ is $\la \ll 2^k$. Choosing $\lambda = 404 (1+k)^3 $ will suffice. Summing first according to $k$ gives the bound
$$
(y+1)e^{-2y}
\sum_{r<j \leq t-1}e^{-2\psi_{j+1}}
\sum_{v \geq 1}
e^{-200v}
(v+1) e^{-(y+\psi_{j+1}-v+c)^2 /j}.
$$
This corresponds to Equation \eqref{eqn: both} up to the constant $c$. It is bounded the same way. 
The constant $c$ does not matter in the end, because  $e^{-(y+\psi_{j+1}+c)^2 /t}\ll e^{-y^2/t}$ since $y=\oo(t)$.
This proves the lemma.
\end{proof}

\subsection{Upper Bound for Theorem \ref{thm:thetamove}}
In this section, we prove the upper bound to Theorem \ref{thm:thetamove}. Note that the convergence is slightly better than the one stated as the $\e\log\log\log T$ term from the definition of the convergence in probability can be replaced at no cost by a function $g(T)$ going to infinity arbitrarily slowly. 
\label{sect: UB thetamove}
\begin{proposition}\label{upper_bd_prop} 
Let $g(T)$ be a function with $g(T)\to +\infty$ arbitrarily slowly, and $\theta=(\log\log T)^{-\alpha}$ with $0<\alpha<1$. We have
\begin{equation}\label{upper_bd}
\PP \left( \max_{|h|\leq (\log T)^\theta} X_T(h) > \sqrt{1+ \theta} \log\log T -\frac{1+ 2 \al}{4 \sqrt{1+\theta}}\log\log\log T +g(T) \right) = \oo(1).
\end{equation}
\end{proposition}

\begin{proof}
We use the short-hand notation $t=\log\log T$, and write, with abuse of notation, $g(t)$ for $g(T)$.
For $\theta = t^{-\al}$, define
\begin{equation}\label{threshold_alpha}
m_{\al}(t)= \sqrt{1+ \theta} \ t - \frac{1+ 2 \al}{4 \sqrt{1+\theta}} \log t  + g(t).
\end{equation}

The probability of the maximum on the interval $[-(\log T)^\theta, (\log T)^\theta]$ is bounded above by a union bound on $(\log T)^\theta$ intervals of length $2$.
The translation invariance of the distribution of the model then gives
\begin{equation}\label{union_bound_theta_1}
\PP\left( \max_{|h|\leq (\log T)^\theta} X_{T} (h) > m_{\al}(t) \right)
\leq
(\log T)^{\theta}
\cdot \PP \left( \max_{|h|\leq 1} X_{T} (h) > m_{\al}(t) \right).
\end{equation}
The recentering $m_\alpha$ can be expressed as $ m_{\al}(t) = m_1(t) + y $ with
$
m_1(t) = t - \frac{3}{4} \log t
$
and
\begin{equation}\label{y_threshold_def}
    y =  (\sqrt{1+ \theta} - 1 ) t - \left( \frac{1+ 2 \al}{4 \sqrt{1+\theta}} - 3/4 \right) \log t  +g(t).
\end{equation}
Note that $y>0$ and $y=\oo(t)$ if $0<\alpha<1$, so that Theorem \ref{thm: right tail} can be applied. It gives
\begin{equation}\label{y_threshold_bound}
  \PP\left( \max_{|h|\leq (\log T)^\theta} X_{T} (h) > m_{\al}(t) \right)
\leq
e^{t\theta}
\cdot ye^{-2y}e^{\frac{-y^2}{t}}.
\end{equation}
From Equation \eqref{y_threshold_def} and the fact that $\theta=t^{-\alpha}$, we deduce that
$$
\begin{aligned}
\frac{y^2}{t}&=(2+\theta-2\sqrt{1+\theta})t +\OO(t^{-\alpha}\log t)\\
2y&=2(\sqrt{1+ \theta} - 1 ) t - (\alpha - 1) \log t +2g(t)  +\OO(t^{-\alpha}\log t),
\end{aligned}
$$
and 
\begin{equation}\label{y_threshold_identity}
\frac{y^2}{t} + 2y = \theta t + (1-\al)\log t +2g(t)  +\OO(t^{-\alpha}\log t).
\end{equation}
We also have by a Taylor expansion
\begin{equation}\label{approx_logy_general}
\log y= \log \theta + \log t - \log 2 + \OO(t^{-\alpha}\log t)= (1- \al) \log t  + \OO(1).
\end{equation}
With these estimates, the right-hand side of Equation \eqref{y_threshold_bound} becomes
\begin{align*}
& \ll \exp \left( {\theta t + \log y - \theta t +(\alpha-1)\log t - 2g(t)} \right) \\
& \ll \exp \left( -2g(t)+\OO(1)\right) = \oo(1).
\end{align*}
This proves the claim for any $\al \in (0,1)$ and concludes the proof.
\end{proof}
\begin{remark}\label{special_case_rmk} In the case $2\al \geq 1$, the above proof can be made more direct. 
It is not hard to check that the weaker estimate
\begin{equation}\label{weaker_estimate}
\PP \left( \max_{|h|\leq 1} X_{T} (h) > m_1(t) + y \right) \ll y e^{-2y}
\end{equation}
is sufficient for the proof.
\end{remark}

\subsection{Upper Bound of Theorem \ref{thm:thetafixed}}
\label{sect: UB thetafixed}
In this section, we prove the upper bound to Theorem \ref{thm:thetafixed}.
\begin{proposition} 
\label{prop: UB thetafixed}
Let $g(T)$ be a function with $g(T)\to +\infty$ arbitrarily slowly, and $\theta>0$ fixed. We have
\begin{equation}
\PP \left( \max_{|h|\leq (\log T)^\theta} X_T(h) > \sqrt{1+ \theta} \log\log  T - \frac{1}{4 \sqrt{1+\theta}}\log\log\log T +g(T) \right) = \oo(1).
\end{equation}
\end{proposition}

We will need the following discretization lemma which states that the maximum of $S_j$ over a neighborhood of size $e^{-j}$ behaves like a single Gaussian random variable of variance $j$. 
\begin{lemma}
\label{lem: discretization}
Let $C>0$ and $1\leq j\leq t$. For any $1<y\leq C j$, we have
$$
\PP\left(\max_{|h|\leq e^{-j}} S_j(h)>y\right)\ll_C \frac{e^{-y^2/j}}{j^{1/2}},
$$
where $\ll_C$ means that the implicit constant depends on $C$.
\end{lemma}

The proof of the proposition is a straightforward union bound using the lemma.
\begin{proof}[Proof of Proposition \ref{prop: UB thetafixed}]
Let
$$
m_0(t)= \sqrt{1+ \theta}\ t - \frac{1}{4 \sqrt{1+\theta}}\log t +g(t).
$$
The interval $[-(\log T)^\theta, (\log T)^\theta]$ can be split into $(\log T)^{1+\theta}$ intervals each of length $2e^{-t}$. Therefore, a union bound gives
$$
\PP \left( \max_{|h|\leq (\log T)^\theta} X_T(h) >m_0(t) \right) 
\leq (\log T)^{1+\theta}\cdot \PP\left(\max_{|h|\leq e^{-t}} X_T (h)>m_0(t)\right).
$$
Lemma \ref{lem: discretization} can be applied with $j=t$ (recall that $S_t=X_T$), and $C=\sqrt{1+\theta}$, yielding
$$
\PP \left( \max_{|h|\leq (\log T)^\theta} X_T(h) >m_0(t) \right) \ll e^{t(1+\theta)}\cdot \frac{e^{-m_0(t)^2/t}}{t^{1/2}}.
$$
This is $\oo(1)$, thereby concluding the proof. 
\end{proof}
Lemma \ref{lem: discretization} is proved by a chaining argument similarly to Lemma \ref{lem: fluct}.
\begin{proof}[Proof of Lemma \ref{lem: discretization}]
The probability can be divided into
$$
\PP\left(\max_{|h|\leq e^{-j}} S_j(h)>y, S_j(0)>y-1\right)+\PP\left(\max_{|h|\leq e^{-j}} S_j(h)>y, S_j(0)\leq y-1\right).
$$
The first term is $\leq \PP\left(S_j(0)>y-1\right)\ll \frac{e^{-y^2/j}}{j^{1/2}}$ for $y>1$. It remains to bound the second term. Summing over the values of $S_j(0)$ gives
\begin{multline}
\PP\left(\max_{|h|\leq e^{-j}} S_j(h)>y, S_j(0)\leq y-1\right) \\
\ll \sum_{x\leq y-1}\PP\left(\max_{|h|\leq e^{-j}} S_j(h)-S_j(0)>y-x, S_j(0)\in I_x\right),
\end{multline}
where the intervals $I_x$ are defined below \eqref{eqn: Bj def}.
This is of the same form as the expression in Equation \eqref{eqn: to prove 2} of Lemma \ref{lem: fluct} with $j+1$ instead of $j$, without the sum over $j$, with $\psi_{j+1}=0$, and with the event $\mathscr B_{j+1}(x)$ being simply $\{S_j(0)\in I_x\}$ (without the recentering). Proceeding the same way using a chaining argument up to Equation \eqref{eqn: chaining almost}  leads to
the bound for $v=y-x$
$$
\sum_{v\geq 1}\sum_{k\geq 0}2^k e^{-\lambda \frac{v}{2(1+k)^2}}\cdot \frac{e^{-(y-v)^2/j}}{j^{1/2}}.
$$
Note that the full ballot estimate \eqref{eqn: Btilde} is not needed here, as only the endpoint is involved. We choose $\lambda=200C(1+k)^3$, so that the above becomes
$$
\sum_{v\geq 1}\sum_{k\geq 0}2^k e^{-100C(1+k)v}e^{-(y-v)^2/j}\ll \sum_{v\geq 1} e^{-100Cv}\cdot \frac{e^{-(y-v)^2/j}}{j^{1/2}}.
$$
The summand is maximal at $v=y-100Cj$. This is negative fo $y\leq Cj$. We conclude that the maximizer is at $v=1$ so that
$$
\PP\left(\max_{|h|\leq e^{-j}} S_j(h)>y, S_j(0)\leq y-1\right)\ll \frac{e^{-y^2/j}}{j^{1/2}}
$$
from which the claim follows.
\end{proof}

\section{Lower Bounds}
\label{sect: LB}

The lower bounds to Theorem \ref{thm:thetafixed} and Theorem \ref{thm:thetamove} are stated in the next two propositions.

\begin{proposition}
\label{prop: LB}
For any $\eps>0$, we have for $\theta\sim (\log\log T)^{-\alpha}$
$$
\PP \left( \max_{|h| \leq (\log T)^\theta } X_T(h) > \sqrt{1+ \theta} \log\log T - \frac{1 + 2 \al }{4 \sqrt{1+\theta}} \log\log\log T  - \eps \log\log\log T \right) = 1-\oo(1).
$$
\end{proposition}

\begin{proposition}\label{prop: LBgeq0} For any $\eps>0$, we have for $\theta>0$ fixed
$$
\PP \left( \max_{|h| \leq (\log T)^\theta } X_T(h) >\sqrt{1+ \theta}\log\log T - \frac{1}{4 \sqrt{1+\theta}} \log\log\log T  - \eps \log\log\log T \right) = 1-\oo(1).
$$
\end{proposition}

Proposition \ref{prop: LB} is proved in the following subsection. The proof of Proposition \ref{prop: LBgeq0} is very similar, and the differences are discussed in Section \ref{sect: LBgeq0}.
\subsection{Proof of Proposition \ref{prop: LB} }
Recall the $\log\log$ notation in Equation \eqref{eqn: loglog}.
Define also 
\begin{equation}
\label{eqn: mu}
\begin{aligned}
\mu&= \sqrt{1+ \theta} - \frac{1 + 2 \al }{4 \sqrt{1+\theta}} \frac{\log t}{t}-\e \frac{\log t}{t}.
\end{aligned}
\end{equation}
Consider the set of points in the interval $[-e^{\theta t}, e^{\theta t}]$ which lie at a distance $e^{-t}$ of each other
$$\mathcal H=[-e^{\theta t}, e^{\theta t}]\cap e^{-t}\mathbb Z.$$
Note that $\#\mathcal H=2e^{(1+\theta) t}$.
Clearly, it is sufficient for the purpose of the lower bound to bound the maximum on $\mathcal H$. 
$$
\PP \left( \max_{h\in\mathcal H } X_{T}(h) >\mu t\right) = 1-\oo(1).
$$

We need to introduce some notations to state the lemmas needed for the proof. Define for a fixed $\delta>0$
\begin{equation}
\label{eqn: Z}
Z_\delta=\{h\in \mathcal H: X_{T}(h)-\mu t \in[0,\delta],X_{K}(h)\leq k\mu+b(k),\ \forall k\leq t\},
\end{equation}
where the {\it linear barrier} $b(k)$ is 
\begin{equation}\label{barrier}
b(k)= \frac{k}{t} +\frac{t^{1-\alpha}}{10}\left(1-\frac{k}{t}\right), \ k\leq t.
\end{equation}
The barrier $b(k)$ is such that $b(0)=\frac{t^{1-\alpha}}{10}$ and $b(t)=1$. 
(The factor $1/10$ could be replaced by any positive number smaller than $1/2$.)
Consider also $\wPP$ defined in Equation \eqref{eqn: two-point biased measure} with $K=1$ and $\lambda'=0$:
\begin{equation}
\label{eqn: bias1}
\frac{\rd \wPP}{\rd \PP}=\frac{e^{\lambda X_{T}(0)}}{\E[e^{\lambda X_{T}(0)}]}.
\end{equation}
From Lemma \ref{lem: MGF}, we have
\begin{equation}
\label{eqn: MGF Z}
\E[e^{\lambda X_{T}(0)}]
=e^{\lambda^2\sigma_t^2/2}, \quad \sigma_t^2=\frac{1}{2}\sum_{p\leq T}\frac{1}{p}=\frac{1}{2}t +\OO(1),
\end{equation}
where Equation \eqref{eqn: mertens} is used to estimate the variance.
The mean is by Equation \eqref{eqn: tilted mean}
\begin{equation}
\label{eqn: tilted mean 2}
\wE[X_{T}(0)]=\lambda\sigma_t^2.
\end{equation}
Therefore, for the tilted mean to be $\mu t$, we take
\begin{equation}
\label{eqn: lambda 1}
\lambda=\frac{\mu t}{\sigma_t^2}=2\mu+\OO(1/t).
\end{equation}
With this in mind, define the recentered partial sums $\overline X_{K}(0)=X_{K}(0)-\wE[X_{T}(0)]$, $k\leq t$. 
It is also convenient to define the event
\begin{equation}
\label{eqn: J}
J(h)=\{\overline{X}_{T}(h)\in[0,\delta],\overline{X}_{K}(h)\leq b(k), \ \forall k\in [\![1,t]\!]\}, \quad h\in \mathcal H.
\end{equation}
\begin{lemma}
\label{lem: first moment}
Let $Z_\delta$ be as in Equation \eqref{eqn: Z} for $\delta>0$ and $\mu$ as in Equation \eqref{eqn: mu}.
We have
$$
\E[Z_\delta]= 2 e^{(1+\theta) t}\cdot  e^{-\mu^2 t^2/(2\sigma_t^2)} \cdot  e^{-\delta\mu t/\sigma_t^2}\cdot \wPP(J(0)).
$$
In particular, this gives the lower bound
\begin{equation}
\label{eqn: Z lower}
\E[Z_\delta]\gg t^{\e} \cdot \delta e^{-2\mu\delta},
\end{equation}
which shows that $\E[Z_\delta]\to\infty$ as $t\to\infty$, for fixed $\delta$.
\end{lemma}

\begin{lemma}
\label{lem: second moment}
\begin{equation}
\label{eqn: Z^2}
\E[Z_\delta^2]\leq \Big\{(2+\oo(1))\ e^{(1+\theta) t}\cdot e^{-\mu^2 t^2/(2\sigma_t^2)} \wPP(J(0))\Big\}^2.
\end{equation}
\end{lemma}

\begin{proof}[Proof of Proposition \ref{prop: LB}]
The Paley-Zygmund inequality implies
$$
\PP \left( Z_\delta\geq 1 \right)
\geq \frac{(\E[Z_\delta])^2}{\E[Z_\delta^2]}=(e^{-4\mu\delta}+\oo(1))\ ,
$$
by Lemmas \ref{lem: first moment} and \ref{lem: second moment}, and the fact that $\mu t/\sigma_t^2\to 2\mu$. The claim follows by taking the limit $T\to\infty$ and then $\delta\to0$ .
\end{proof}

\begin{proof}[Proof of Lemma \ref{lem: first moment}]
By linearity of expectation and the change of measure \eqref{eqn: bias1}, the expectation can be written, with the choice $\lambda=\mu t/\sigma_t^2$, as
\begin{equation}
\label{eqn: EZ}
\begin{aligned}
\E[Z_\delta]&= (\#\mathcal H)\cdot \E[e^{\lambda X_{T}(0)}]\cdot  e^{-\lambda \wE[X_T(0)]}\cdot \wE[e^{-\lambda \overline X_{T}(0)}\1_{J(0)}]\\
&= 2e^{(1+\theta) t}\cdot e^{\mu^2 t^2/(2\sigma_t^2)}\cdot e^{-\mu^2 t^2/\sigma_t^2}\cdot e^{-\delta\mu t/\sigma_t^2}\cdot \wPP(J(0))\\
&=2e^{(1+\theta) t}\cdot e^{-\mu^2 t^2/(2\sigma_t^2)}\cdot  e^{-\delta\mu t/\sigma_t^2}\cdot \wPP(J(0))
\end{aligned}
\end{equation}
where we used Equations \eqref{eqn: MGF Z}, \eqref{eqn: tilted mean 2} and \eqref{eqn: lambda 1}. This proves the first claim. 
For the second claim, the definition of $\mu$ gives
$$
e^{-\mu^2 t^2/(2\sigma_t^2)}
\ll e^{-(1+\theta)t}\cdot t^{\alpha+1/2+\e}.
$$
It remains to apply Proposition \ref{thm: ballot linear LB} to $\wPP(J(0))$:
\begin{equation}
\label{eqn: PJ}
\wPP(J(0))\gg \delta \frac{b(t)\cdot b(0)}{t^{3/2}}\gg \delta t^{-1/2-\alpha}.
\end{equation}
\end{proof}

\begin{proof}[Proof of Lemma \ref{lem: second moment}]

The second moment of $Z_\delta$ can be written as a sum over pairs of $h$'s
$$
\E[Z_\delta^2]=\sum_{h,h'}\PP(J(h)\cap J(h'))\leq(\#\mathcal H)\cdot \sum_{|h|\leq (\log T)^\theta} \PP(J(0)\cap J(h)),
$$
where the last equality is by the symmetry in $h$'s. 
The sum is split into three terms: $|h|\leq 1$, $|h|>e^{t\theta/2}$, and $1<|h|\leq e^{t\theta/2}$. The dominant term is the latter.

For the sum over $|h|> e^{t\theta/2}$, we use the biased measure $\wPP$  in Equation \eqref{eqn: two-point biased measure} with $\lambda=\lambda'$ and $K=1$
\begin{equation}
\label{eqn: two-point biased measure 2}
\frac{\rd \wPP}{\rd \PP}=\frac{e^{\lambda (X_{T}(0)+X_{T}(h))}}{\E[e^{\lambda (X_{T}(0)+X_{T}(h))}]}.
\end{equation}
Lemma \ref{lem: MGF} gives for the Laplace transform
\begin{equation}
\label{eqn: MGF 2}
\begin{aligned}
\E[e^{\lambda (X_{T}(0)+X_{T}(h))}]
&=\exp\Big(\frac{\lambda^2}{2}\sum_{p\leq T}\frac{1+\cos(|h|\log p)}{p}\Big)\\
&=\exp\Big(\lambda^2 \sigma_T^2+\OO(|h|^{-1})\Big).
\end{aligned}
\end{equation}
In particular, the tilted mean is by Equation \eqref{eqn: tilted mean}
\begin{equation}
\label{eqn: tilted mean 3}
\wE[X_T(0)]=\lambda\sigma_t^2+\OO(|h|^{-1}).
\end{equation}
We choose $\wE[X_T(0)]=\mu t$ so that
\begin{equation}
\label{eqn: lambda 2}
\lambda=\frac{\mu t}{\sigma_t^2} +\OO(e^{-t\theta/2})=2\mu +\OO(1/t).
\end{equation}
Lemma \ref{lem: decoupling} gives 
\begin{equation}
\label{eqn: decoupling}
\begin{aligned}
\wPP(J(0)\cap J(h))
&=(1+\OO(e^{-t\theta/2})) \wPP(J(0))\cdot \wPP(J(h)) +\OO(\exp(-ce^{t\theta/4}))\\
&=(1+\oo(1)) \wPP(J(0))\cdot \wPP(J(h)),
\end{aligned}
\end{equation}
where Equation \eqref{eqn: PJ} is used to get the second equality.
As before, define the recentered process $\overline X_{K}=X_{K}-\wE[X_K]$, $k\in[\![1,t]\!]$. 
With this notation, the contribution of $|h|>e^{t\theta/2}$ to $\E[Z_\delta]$ is
\begin{equation}
\label{eqn: dominant}
\begin{aligned}
(\#\mathcal H)\cdot \sum_{|h|>e^{t\theta/2}}\PP(J(0)\cap J(h))
&= e^{\lambda^2\sigma_t^2} \cdot e^{-2\lambda\wE[X_T(0)]} \cdot (\#\mathcal H)\cdot \sum_{{|h|>e^{t\theta/2}}} \wPP( J(0)\cap J(h))\\
&=(2+\oo(1))^2 \ e^{-\mu^2t^2/\sigma_t^2} \cdot e^{2t(1+\theta)} \cdot \wPP(J(0))^2,
\end{aligned}
\end{equation}
where in the first inequality we bounded the factor $e^{-\lambda (\overline X_{T}(0)+\overline X_{T}(h))}$ in the expectation with respect to $\wPP$ from above by $1$ (on the event $J(0)\cap J(h)$). In the second equality, we used Equations \eqref{eqn: tilted mean 3}, \eqref{eqn: decoupling}, and the fact that there are $(2+\oo(1))e^{(1+\theta)t}$ terms in the sum.
This is (to leading order) equal to  the right-hand side of Equation \eqref{eqn: Z^2}.
It remains to prove that the other contributions are negligible compared to $(\E[Z_\delta])^2$.

For the case $1<|h|\leq e^{t\theta/2}$, we proceed as in Equation \eqref{eqn: dominant}. The number of $h'$ in the sum is now $\ll e^{t+t\theta/2}$ instead of being $\gg e^{t(1+\theta)}$. 
Moreover, the additive error term in \eqref{eqn: decoupling} is $\OO(1)$. Since $\wPP(J(0))\gg e^{-t\theta/4}$ by \eqref{eqn: PJ}, we get
\begin{equation}
\label{eqn: not so dominant}
\begin{aligned}
(\#\mathcal H)\cdot \sum_{1<|h|\leq e^{t\theta/2}}\PP(J(0)\cap J(h))
\ll
e^{-t\theta/4}\Big( e^{-\mu^2t^2/\sigma_t^2} \cdot e^{2t(1+\theta)} \cdot \wPP(J(0))^2\Big).
\end{aligned}
\end{equation}
This is $\oo((\E[Z_\delta])^2)$ by Lemma \ref{lem: first moment}.

It remains to bound the contribution of $|h|\leq 1$. 
It is convenient to split into the intervals $e^{-k-1}<|h|\leq e^{-k}$ for $0\leq k\leq t$. 
For a given $k$, the contribution to $\E[Z_\delta^2]$, that we denote by $\Big(\E[Z_\delta^2]\Big)_k$ is
$$
\Big(\E[Z_\delta^2]\Big)_k\ll (\#\mathcal H)\cdot \sum_{e^{-k-1}< |h|\leq e^{-k}}\PP(J(0)\cap J(h)).
$$
We handle two ranges of $k$'s differently:
$$
(I):\quad 0\leq k\leq t-t^{\alpha+\e/100},  \qquad (II): \quad t-t^{\alpha+\e/100}< k\leq t.
$$
In the first range, most of the barrier is not needed to show that the contribution is negligible compared to 
$(\E[Z_\delta])^2\gg t^{2\e}$.  Only the value at the ``branching time $k$'' is necessary. 
This is because the barrier is too high at these times to affect the probability of the walks. 
For technical reason, we need the boundary case $k=t-t^{\alpha+\e/100}$ to be slightly smaller than $t-t^\alpha$, where the barrier becomes of order one.

To handle the additive error term coming out from the decoupling in Lemma \ref{lem: decoupling}, it is necessary to split the random walk $X_T$
into $X_{K'}$ and $X_{K',T}$ where $K'$ is slightly larger than $K$. We take $k'=k+r$ where $r=10\log\log t$.
We consider the biased measure $\wPP$ now defined separately for $X_{K'}$ and $X_{K',T}$
\begin{equation}
\label{eqn: two-point biased measure 3}
\frac{\rd \wPP}{\rd \PP}=\frac{e^{\lambda_1(X_{K'}(0)+X_{K'}(h))}}{\E[e^{\lambda_1 (X_{K'}(0)+X_{K'}(h))}]}\cdot \frac{e^{\lambda_2 (X_{K',T}(0)+X_{K',T}(h))}}{\E[e^{\lambda_2 (X_{K',T}(0)+X_{K',T}(h))}]}.
\end{equation}
We pick $\lambda_1=\mu$ and $\lambda_2=2\mu$ to reflect the fact that $X_{K'}(0)$ and $ X_{K'}(h)$ are almost perfectly correlated for $|h|\approx e^{-k}$.
The Laplace transforms are easily computed using Lemma \ref{lem: MGF} and Lemma \ref{lem: cos sum}:
\begin{equation}
\label{eqn: MGF 3}
\E[e^{\mu (X_{K'}(0)+X_{K'}(h))}]\ll e^{\mu^2k'}=e^{\mu^2(k+r)}\qquad \E[e^{2\mu (X_{K,T}(0)+X_{K,T}(h))}]\ll e^{2\mu^2(t-k')}= e^{2\mu^2(t-k-r)}.
\end{equation}
The tilted mean are given by Equation \eqref{eqn: tilted mean}  together with Lemma \ref{lem: cos sum}:
$$
\wE[X_{K'}(0)]=\wE[X_{K'}(h)]=\mu k'+\OO(1)\quad \wE[X_{K',T}(0)]=\wE[X_{K',T}(h)]=\mu (t-k')+\OO(e^{-e^{k'/2}}).
$$
Again, the recentered sums are  $\overline X_{K'}=X_{K'}-\wE[X_K']$ and $\overline X_{K',T}=X_{K',T}-\wE[X_{K',T}]$.

{\it Case $(I)$.} The probability is split using the value of $\overline X_{K'}$ at $0$ and $h$. Note that these values are bounded by $b(k')$. 
Since there are approximately $e^{t-k}$ $h$'s with $|h|\approx e^{-k}$ in $\mathcal H$, we get
\begin{equation}
\label{eqn: Z2 k split}
\begin{aligned}
&\Big(\E[Z_\delta^2]\Big)_k\\
&\ll (\#\mathcal H)\cdot e^{t-k}\cdot \sum_{q_1,q_2\leq b(k')}
\PP(\overline{X}_{K'}(0)\in (q_1-1,q_1], \overline{X}_{K'}(h)\in (q_2-1,q_2])\\
&\hspace{3.5cm}\cdot \PP(\overline{X}_{K',T}(0)+q_1\in [0,\delta+1], \overline{X}_{K',T}(h)+q_2\in [0,\delta+1])\\
&\ll  e^{t(1+\theta)+(t-k)}\cdot  \sum_{q_1,q_2\leq b(k+r)} e^{-\mu(q_1+q_2)}e^{-\mu^2(k+r)}\cdot e^{2\mu(q_1+q_2)}e^{-2\mu^2(t-k-r)}\\
&\ll  e^r\cdot e^{-\theta(t-k)}\cdot t^{(\tfrac{k}{t}+2\tfrac{t-k}{t})(\alpha+1/2+\e)} \sum_{q_1,q_2\leq b(k+r)} e^{\mu(q_1+q_2)}.
\end{aligned}
\end{equation}
In the second inequality, we changed measure to $\wPP$ using \eqref{eqn: MGF 3}, and simply dropped the probabilities of the events under $\wPP$ keeping only the normalizing constants. 
In the third inequality, we simply replaced the definition of $\mu$ given in \eqref{eqn: mu}. 
Note that  $e^{-\theta(t-k)}\cdot t^{(\tfrac{k}{t}+2\tfrac{t-k}{t})(\alpha+1/2+\e)}\ll e^{-\frac{9}{10}\theta(t-k)}$ for $k\leq t-t^{\alpha+\e/100}$.
Moreover, by definition of $b(k)$, we have $b(j)\leq \theta\frac{(t-j)}{10}+1$ for all $j$, therefore the above is
\begin{equation}
\label{eqn: Z2 k}
\Big(\E[Z_\delta^2]\Big)_k\ll e^r\cdot e^{-\frac{9}{10}\theta(t-k)}\cdot e^{2\mu b(k+r)}\ll e^r\cdot e^{-\frac{\theta}{2}(t-k)}.
\end{equation}
The sum over the range of $k$ considered then yields
$$
\sum_{1\leq k\leq t-t^{\alpha+\e/100}}\Big(\E[Z_\delta^2]\Big)_k\ll  e^r\cdot \theta e^{-\frac{\theta t^{\alpha+\e/100}}{2}}\ll e^{r}.
$$
This is much smaller than $(\E[Z_\delta])^2)\gg  t^{2\e} \cdot \delta^2 e^{-4\mu\delta}$, by Equation \eqref{eqn: Z lower} and the choice of $r$.

{\it Case $(II)$.} It remains to estimate the contribution of $t-t^{\alpha+\e/100}<k\leq t$.
This is done as in Equation \eqref{eqn: Z2 k split} keeping now the probabilities after the change of measure. 
This gives
$$
\begin{aligned}
\Big(\E[Z_\delta^2]\Big)_k
&\ll  e^r\cdot e^{-\theta(t-k)}\cdot t^{(\tfrac{k}{t}+2\tfrac{t-k}{t})(\alpha+1/2+\e)}\cdot\\
&\hspace{4cm}\sum_{q_1,q_2\leq b(k)} e^{\mu(q_1+q_2)}\wPP(A_0(q_1)\cap A_h(q_2))\cdot \wPP(B_0(q_1)\cap B_h(q_2)),
\end{aligned}
$$
where
\begin{equation}
\label{eqn: AB}
\begin{aligned}
A_h(q)&=\{\overline X_{K'}(h)\in (q-1,q],\overline X_{L}(h)\leq b(\ell),\ \forall \ell\in [\![1,k']\!]\}\\
B_h(q)&=\{\overline X_{K',T}(h)+q\in [0,\delta+1],\overline X_{K',L}(h)+q\leq b(\ell),\ \forall \ell \in [\![k'+1,t]\!]\}.
\end{aligned}
\end{equation}
The idea is that $X_{L}(0)$ and $ X_{L}(h)$ should be essentially equal for $\ell\leq k'$, whereas $X_{L,T}(0)$ and $ X_{L,T}(h)$ are essentially independent for $\ell>k'$. 
To obtain an upper bound we can simply drop $A_h(q_2)$ in the inequality. This is expected to be almost sharp since $X_{L}(0)$ and $ X_{L}(h)$ should be essentially equal for $\ell\leq k$.
For the decoupling, since $k'=k+r$ slightly larger than $k$, the additive error term in Lemma \ref{lem: decoupling} will be absorbed. 
More precisely, in order to estimate $\wPP(B_0(q_1)\cap B_h(q_2))$, we apply Lemma \ref{lem: decoupling}, Corollary \ref{cor: var 1/2} and Equation \ref{ballot_inequality2} to get
\begin{equation}
\label{eqn: B}
\begin{aligned}
\wPP(B_0(q_1)\cap B_h(q_2))&\ll \wPP(B_0(q_1))\cdot \wPP(B_h(q_2)) +\OO(e^{-c\sqrt{e^r}})\\
&\ll \frac{1}{(t-k-r)^{3}}\prod_{i=1,2}(b(k+r)-q_i+1),
\end{aligned}
\end{equation}
by the choice of $r$ and the fact that $b(t)=1$. 
The probability $\wPP(A_0(q_1))$ is bounded using Corollary \ref{cor: var 1/2} and Proposition \ref{thm: ballot linear UB} to give
\begin{equation}
\label{eqn: A}
\wPP(A_0(q_1))\ll \frac{(b(k+r)-q_1+1)(b(0)+1)}{(k+r)^{3/2}}\ll\frac{(b(k+r)-q_1+1)\cdot t^{1-\alpha}}{k^{3/2}},
\end{equation}
where we use the fact that $k$ is much larger than $r$ in the range considered. 
(The increments here do not have variance exactly $1/2$, but can be simply fixed by considering the slightly larger event where the barrier start at $r$,
and by conditioning on the position at time $r$. The error term in Corollary \ref{cor: var 1/2} can then be easily absorbed.)

Putting Equations \eqref{eqn: A} and \eqref{eqn: B} together yield
$$
\begin{aligned}
&\Big(\E[Z_\delta^2]\Big)_k\ll e^r\cdot e^{-\theta(t-k)}\cdot \frac{t^{(\tfrac{k}{t}+2\tfrac{t-k}{t})(\alpha+1/2+\e)+(1-\alpha)}}{k^{3/2}(t-k)^3}\cdot\\
&\sum_{q_1,q_2\leq b(k+r)} e^{\mu(q_1+q_2)} (b(k+r)-q_1+1)\cdot  (b(k+r)-q_1+1)(b(k+r)-q_2+1).
\end{aligned}
$$
The three factors in parentheses are smaller than $t^{\e/2}$ since $b(k+r)\ll t^{\e/100}$. 
Moreover, we have similarly as in Case $(I)$: $e^{-\theta(t-k)}e^{2\mu b(k+r)}\ll e^{-\frac{9}{10}\theta(t-k)}\ll1$ for all $k$ in the range.
The above becomes
$$
\Big(\E[Z_\delta^2]\Big)_k\ll  e^r\cdot t^{\e/2}\cdot \frac{t^{(\tfrac{k}{t}+2\tfrac{t-k}{t})(\alpha+1/2+\e) +(1-\alpha)}} {k^{3/2}(t-k)^3}.
$$
The numerator can also be simplified:
$$
t^{(\tfrac{k}{t}+2\tfrac{t-k}{t})(\alpha+1/2+\e) +(1-\alpha)}= t^{3/2+\e}\cdot t^{(1-\tfrac{k}{t})(\alpha+1/2+\e)}.
$$
We are left with
$$
\Big(\E[Z_\delta^2]\Big)_k\ll \frac{ e^r \cdot t^{3/2}\cdot t^{3\e/2}\cdot t^{(1-\tfrac{k}{t})(\alpha+1/2)}} {k^{3/2}(t-k)^3}\ll   \frac{ e^r\cdot t^{(1-\tfrac{k}{t})(\alpha+1/2+\e)}} {(t-k)^3},
$$
since $k^{3/2}>(t-t^{\alpha+\e/100})^{3/2}\gg t^{3/2}$.
It remains to sum over $k$. After the change of index $\hat k=t-k$, the contribution of the range is
$$
\sum_{k=t-t^{\alpha+\e/100}+1}^{t}\Big(\E[Z_\delta^2]\Big)_k
\ll e^r\cdot t^{3\e/2}\cdot \sum_{\hat k=1}^{t^{\alpha+\e/100}}  \frac{  t^{\frac{\hat k}{t}(\alpha+1/2+\e)}}{{\hat k}^3}\ll  e^r\cdot t^{3\e/2}\cdot \sum_{\hat k=1}^{\infty}  \frac{1}{{\hat k}^3}\ll t^{7\e/4}.
$$
This proves that this contribution is $\oo(\E[Z_\delta])^2)$ by  Equation \eqref{eqn: Z lower} and concludes the proof of the lemma.
\end{proof}

\subsection{Proof of Proposition \ref{prop: LBgeq0}}
\label{sect: LBgeq0}
We now explain which modifications in the proof of Proposition \ref{prop: LB} are necessary to prove Proposition \ref{prop: LBgeq0}. Note that the case $\theta$ greater than zero is actually simpler as we do not need the ballot-type estimates used in Case $(II)$ in the proof of  Lemma  \ref{lem: second moment}.

\begin{proof}[Proof of Proposition \ref{prop: LBgeq0}]
As seen in the proof of Proposition \ref{prop: LB} it suffices to proof a lower bound on the first moment and an upper bound on the second moment of $Z_\delta$.
Adapting the barrier in \eqref{barrier} to the case $\theta>0$ we choose
\begin{equation}\label{barriergeq0}
b(k)= \frac{k}{t} +\frac{\theta t }{10}\left(1-\frac{k}{t}\right), \ k\leq t.
\end{equation}
Moreover, in the definition of $\mu$ we set $\alpha=0$. 

We start by showing a lower bound on $\E[Z_\delta]$. The first bound  in  Lemma \ref{lem: first moment} (and its proof) still hold, with $\alpha=0$, without any modification. Noting that the barrier in \eqref{barriergeq0} is linear and of order $t$ at the beginning we have
$$
\wPP(J(0))\gg t^{-1/2}.
$$
The bound in \eqref{eqn: Z lower} follows.
Next, we turn to the upper bound of $\E[Z_\delta^2]$. 
Equations \eqref{eqn: dominant} and \eqref{eqn: not so dominant} hold as is.

It remains to bound the terms with $e^{-k-1}\leq |h|\leq e^{-k}$ for $k\geq 0$. 
We proceed in a more straightforward way than the proof of Lemma  \ref{lem: second moment}. We split the walk $X_T$ into $X_{K}$ and $X_{K,T}$. 
(There is no need for an extra spacing $K'$ to allow for a better error in the decoupling here.)
We have
\begin{equation}
\label{eqn: Z2 k splitgeq0}
\begin{aligned}
\Big(\E[Z_\delta^2]\Big)_k&\ll (\#\mathcal H)\cdot e^{t-k}\cdot \sum_{q_1,q_2\leq b(k)}
\PP(\overline{X}_{K}(0)\in (q_1-1,q_1], \overline{X}_{K}(h)\in (q_2-1,q_2])\\
&\hspace{3.5cm}\cdot \PP(\overline{X}_{K,T}(0)+q_1\in [0,\delta+1], \overline{X}_{K,T}(h)+q_2\in [0,\delta+1])\\
&\ll e^{t(1+\theta)+(t-k)}\cdot  \sum_{q_1,q_2\leq b(k)} \frac{1}{\sqrt{k}}e^{-\mu(q_1+q_2)}e^{-\mu^2k}\cdot e^{2\mu(q_1+q_2)}e^{-2\mu^2(t-k)} \\
&\ll e^{-\theta(t-k)}\cdot t^{(\tfrac{k}{t}+2\tfrac{t-k}{t})(1/2+\e)} \sum_{q_1,q_2\leq b(k)} \frac{1}{\sqrt{k}}e^{\mu(q_1+q_2)}.
\end{aligned}
\end{equation}
In the second line, we changed measure to $\wPP$,  and instead of simply dropping events under $\wPP$ we kept a factor of order $1/\sqrt{k}$ from the first probability.  In the third line, we again simply replaced the definition of $\mu^2$. 
By definition of $b(k)$, we have $b(k)\leq \theta\frac{(t-k)}{10}+1$, therefore for $k<t-C\log t$  (for some constant $C$ large enough),
\begin{equation}
\label{eqn: Z2 kgeq0}
\Big(\E[Z_\delta^2]\Big)_k\ll  e^{-\frac{9}{10}\theta(t-k)}\cdot e^{2\mu b(k)}\ll  e^{-\frac{\theta}{2}(t-k)},
\end{equation}
where we absorbed the $t^{1/2}$ in the exponential, and we bounded $\mu$ by $2$.
The sum over $k$ is thus of order one. 
For $k>t-C\log t$ we observe that $t^{(\tfrac{k}{t}+2\tfrac{t-k}{t})(1/2+\e)}/\sqrt{k}$ is of order $t^\e$. 
Therefore, we get in this range
$$
\sum_{k>t-C\log t}\Big(\E[Z_\delta^2]\Big)_k\ll t^{\e}\sum_{k>t-C\log t}e^{-\theta(t-k)}\cdot e^{2\mu b(k)}\ll t^\e.
$$
The upper bound on the second moment then follows as in the Lemma \ref{lem: second moment}. This  concludes the proof of Proposition \ref{prop: LBgeq0}. 
\end{proof}

\begin{appendix}

\section{Gaussian Ballot Theorem}
In the propositions below, it is useful to define for $0<\delta\leq 1$ the intervals $I_x=(x,x+\delta]$ for $x>0$, $I_x=(x-\delta,x]$ if $x<0$, and $I_x=[-\delta,\delta]$ if $x=0$.
In the first proposition, we allow the increments of the random walk to have different variances. 
In the second and third, all increments have variance $1/2$.

\begin{proposition}[Ballot theorem with linear barrier: Lower bound]\label{thm: ballot linear LB}
Let $(S_j,j\geq 1)$ be a random walk with Gaussian increments of mean $0$ and variance $\sigma_j^2$ with $c^{-1}<\sigma_j<c$ for some fixed $c\geq 1$. Define the linear barrier 
$$ 
b(j) = aj+b(0).
$$
Then, for any $j\geq 1$, $(b(j)-x)\cdot b(0)\leq j$, we have
\begin{equation}\label{ballot_inequality_lin}
\PP \left( \forall 1\leq \ell \leq j, \ S_\ell \leq  b(\ell), \ S_j \in I_x\right)
\gg \frac{b(0) ( b(j) -x)}{ j^{3/2} } \delta e^{- \frac{cx^2}{j} }.
\end{equation}
\end{proposition}
\begin{proof}
The probability is bounded below by the probability for a Brownian motion on the continuous-time interval, see for example Lemma 6.2 in \cite{Webb2011}.
\end{proof}

\begin{proposition}[Ballot theorem with linear barrier: Upper bound]\label{thm: ballot linear UB}
Let $(S_j,j\geq 1)$ be a random walk with Gaussian increments of mean $0$ and variance $1/2$. Define the linear barrier 
$$ 
b(j) = aj+b(0).
$$
Then, for any $j\geq 1$, $b(0)> 0$ and $x \leq  b(j)$, we have
\begin{equation}\label{ballot_inequality2}
\PP \left( \forall 1\leq \ell \leq j, \ S_\ell \leq b(\ell), \ S_j \in I_x\right)
\ll
\frac{(b(0)+1) (b(j) -x+1 )}{ j^{3/2} }.
\end{equation}
\end{proposition}
\begin{proof}
This is a standard ballot theorem with linear barrier, see Lemma 6.2 in \cite{Webb2011}.
\end{proof}

\begin{proposition}[Ballot theorem with logarithmic barrier]\label{thm_ballot}
Let $(S_j,j\leq t)$ be a random walk with Gaussian increments of mean $0$ and variance $1/2$. Define the logarithmic barrier 
$$ 
\psi_j =  \log \left(  \min (j,t-j) \right).
$$
Let $t/\log t\leq k\leq t$, $y=\oo(t/\log t)$, $-20k<x\leq \psi_k$.
For $r=y$, we have
\begin{equation}\label{ballot_inequality_log}
\PP \left( \forall r<j \leq k, \ S_j\leq y + \psi_j, \ S_k \in I_x, |S_r|\leq 3r \right)
\ll
\frac{(y+1) (y + \psi_k -x+1 )}{ k^{3/2} } e^{- \frac{x^2}{k} }.
\end{equation}
\end{proposition}
\begin{proof}
See Propositions 4 and 5 in \cite{ArgBouRad20}. Proposition 4 there was stated for $k\geq t/2$ but holds {\it verbatim} for $k\geq t/\log t$ by applying Proposition 5. 
The condition $r=y=\oo(t/\log t)$ ensures that $r$ is small compared to $k$ so that $(k-r)^{3/2}\sim k^{3/2}$. The lower bound on $x$ is to make sure the distance to the barrier at time $k$ is $\ll k$. 
Finally, the bound $|S_r|\leq 3r$ is such that the initial distance at time $r$ is no more than the order of $y$, so that the factor $(y+1)$ (measuring the distance between the barrier and the starting point at the initial time $r$) on the right side still holds. 
\end{proof}

\end{appendix}

\bibliographystyle{imsart-nameyear}
\bibliography{zeta_interpol_bib}

\end{document}